\documentclass{amsart}
\author{Dustin Ross and Zhengyu Zong}
\address{University of Michigan, Department of Mathematics, Ann Arbor, MI 48109, USA}
\email{dustyr@umich.edu}
\address{Columbia University, Department of Mathematics, New York, NY 10027, USA}
\email{zz2197@math.columbia.edu}

\usepackage{mathrsfs}
\usepackage{amsfonts}
\usepackage {color,graphicx,psfrag, verbatim,amssymb,amscd,enumerate,subfigure,tikz,ytableau}

\usepackage{amsmath}
\usepackage{amscd}
\usepackage{amssymb}

\newcommand{\bq}{{\mathrm q}}
\newcommand{\fq}{\mathfrak{q}}
\newcommand{\cal}{\mathcal}

\newcommand{\bC}{{\mathbb C}}
\newcommand{\bE}{{\mathbb E}}

\newcommand{\bN}{{\mathbb N}}
\newcommand{\bP}{{\mathbb P}}

\newcommand{\bZ}{{\mathbb Z}}

\newcommand{\cB}{{\cal B}}

\newcommand{\cD}{{\cal D}}

\newcommand{\cG}{{\cal G}}

\newcommand{\cM}{{\cal M}}
\newcommand{\cN}{{\cal N}}
\newcommand{\cO}{{\cal O}}

\newcommand{\cT}{{\cal T}}

\newcommand{\cU}{{\cal U}}

\newcommand{\cX}{{\cal X}}

\newcommand{\Mbar}{\overline{\cM}}

\DeclareMathOperator{\Aut}{Aut}

\DeclareMathOperator{\rk}{rk}
\DeclareMathOperator{\Cont}{Cont}

\newcommand{\tz}{\tilde{z}}

\newtheorem{theorem}{Theorem}[section]

\newtheorem{lemma}[theorem]{Lemma}
\newtheorem{conjecture}{Conjecture}[section]
\newtheorem{corollary}{Corollary}[section]

\theoremstyle{remark}
\newtheorem{remark}{Remark}[section]
\newtheorem{notation}{Notational Convention}[section]
\theoremstyle{definition}

\newtheorem{definition}{Definition}[section]

\begin{document}
\title{Cyclic Hodge Integrals and Loop Schur Functions}

\maketitle

\begin{abstract}
We conjecture an evaluation of three-partition cyclic Hodge integrals in terms of loop Schur functions.  Our formula implies the orbifold Gromov-Witten/Donaldson-Thomas correspondence for toric Calabi-Yau threefolds with transverse $A_n$ singularities.  We prove the formula in the case where one of the partitions is empty, and thus establish the orbifold Gromov-Witten/Donaldson-Thomas correspondence for local toric surfaces with transverse $A_n$ singularities.
\end{abstract}

\section{Introduction}

\subsection{Statement of Results}

This paper investigates the relation between the Gromov-Witten theory and the Donaldson-Thomas theory of Calabi-Yau $3$-orbifolds with transverse $A_n$ singularities.  The Gromov-Witten partition function $GW^\bullet(\cX;{\bf x},u,v)$ is a formal series encoding intersection numbers on the moduli stack of orbifold stable maps to $\cX$ and the Donaldson-Thomas partition function $DT'(\cX;{\bf q},v)$ is a formal series encoding intersection numbers on the Hilbert scheme of substacks in $\cX$.  We study the following conjecture, first suggested by Bryan-Cadman-Young \cite{bcy:otv}.

\begin{conjecture}\label{gwdt}
There is an explicit change of variables ${\bf q}({\bf x}, u)$ such that \[DT'(\cX;{\bf q},v)=GW^\bullet(\cX;{\bf x},u,v).\]
\end{conjecture}

The variables $\bf{q}$ and $\bf{x}$ are subdivided into sets of variables corresponding to each singular line in $\cX$; in Section \ref{sec:conjecture} we describe the precise change of variables between these sets of variables.

In the toric case, both the GW and DT partition functions can be decomposed into contributions defined locally at each torus fixed point \cite{r:lgoa,bcy:otv}, the so-called orbifold topological vertex.  One motivation for the topological vertex is to use this decomposition to reduce global phenomena to the local setting.  We carry out such a reduction explicitly for Conjecture \ref{gwdt}.

\begin{theorem}[Theorem \ref{thm:imply}]
In the toric setting, Conjecture \ref{gwdt} is implied by an explicit correspondence between the contributions at each torus fixed point.
\end{theorem}

In \cite{r:lgoa}, the local GW partition function at each torus fixed point was computed explicitly in terms of three-partition cyclic Hodge integrals on the moduli stack of stable maps into the classifying stacks $\cB\bZ_n$.  These local contributions are indexed by triples of conjugacy classes $(\mu^1,\mu^2,\mu^3)$ in the generalized symmetric groups $\bZ_n\wr S_{|\mu^i|}$.  In \cite{bcy:otv}, the local DT contributions at each torus fixed point were computed in closed form in terms of loop Schur functions.  The local DT contributions are indexed by triples of irreducible representations $(\sigma^1,\sigma^2,\sigma^3)$ in $\bZ_n\wr S_{|\sigma^i|}$.  The local GW/DT correspondence of Theorem \ref{thm:imply} relates the local partition functions through a change of variables and by identifying the indexing triples via the character table of $\bZ_n\wr S_d$.  The main result of the current paper is the following.

\begin{theorem}[Theorem \ref{thm:main}]
The local GW/DT correspondence holds whenever $\mu^i=\sigma^i=\emptyset$ for at least one $i$.
\end{theorem}

As a corollary of Theorems \ref{thm:imply} and \ref{thm:main}, we obtain a complete proof of Conjecture \ref{gwdt} when $\cX$ is a local toric surface.

Our approach to Theorem \ref{thm:main} is to prove the comparison directly in two steps: we first use fixed point localization to develop a set of relations which determine the GW vertex, then we argue combinatorially that the DT vertex also satisfies these relations.  An alternative approach to the orbifold GW/DT correspondence is to reduce it to the smooth toric case where the GW/DT correspondence was proved in \cite{moop:gwdtc}.  We intend to study this approach in future work by utilizing the open crepant resolution correspondence of \cite{bcr:craos}.

\subsection{Context and Motivation}

Throughout the last 20 years, Gromov-Witten theory has attracted a considerable amount of attention in mathematics and physics.  Defined as intersection numbers on the moduli stack of stable maps, GW invariants are rational numbers in general.  It has long been expected that GW invariants can be expressed in terms of integer curve counts, and Donaldson-Thomas theory conjecturally provides an explanation for this expectation.  It was originally conjectured by Maulik-Nekrasov-Okounkov-Pandharipande \cite{mnop:gwdti} that the natural generating functions encoding GW and DT theory are equivalent after a change of variables in the formal parameters.

In recent years, both GW and DT theory have been generalized to orbifold targets \cite{cr:ogwt,bcy:otv}.  One might naturally expect the GW/DT correspondence also to extend in some cases.  In the case of toric varieties with transverse $A_n$ singularities, the orbifold correspondence has been studied previously by the authors \cite{r:lgoa,z:gmvf,rz:ggmv}, culminating in a complete proof for the case of local lines.  The current paper generalizes those results.

In the rest of the introduction, we describe more precisely the objects of study in order to make Conjecture \ref{gwdt} precise.

\subsubsection{Targets} Let $\cX$ be a Calabi-Yau (CY) $3$-orbifold which is either projective or toric.  We say that $\cX$ has \textit{transverse $A_n$ singularities} if the nontrivial orbifold structure of $\cX$ has cyclic isotropy and is supported on a union of disjoint lines.  Let $L=\{l\}$ be the support of the nontrivial orbifold structure and fix once and for all an isomorphism of each isotropy group with $\bZ_{n(l)}$.

\subsubsection{Gromov-Witten Theory}

GW invariants are virtual intersection numbers on $\Mbar_{g,m}(\cX,\beta)$, the Kontsevich moduli stack of genus $g$, $m$ pointed stable maps to $\cX$ of degree $\beta\in H_2(\cX,\bZ)$.  $\Mbar_{g,m}(\cX,\beta)$ decomposes into substacks indexed by the possible orbifold structure at each marked point.  More precisely, if $\gamma=\{\gamma^l\}$ with each $\gamma^l$ a tuple of elements in $\bZ_{n(l)}=\langle \xi_{n(l)} \rangle$ with $m_{l,k}$ entries equal to $\xi_{n(l)}^k$, then we denote by $\Mbar_{g,\gamma}(\cX,\beta)\subset\Mbar_{g,m}(\cX,\beta)$ the component in the moduli space with orbifold structure at the marked points on the source curve prescribed by $\gamma$.  If $\cX$ is projective, the GW partition function is defined by
\[
GW^{\bullet}({\bf x},u,v):=\exp\left(GW({\bf x},u,v)\right)
\]
where
\[
GW({\bf x},u,v):=\sum_{\substack{\beta\neq 0 \\ g,\gamma}}\left(\int_{[\Mbar_{g,\gamma}(\cX,\beta)]^{vir}}1\right)\frac{{\bf x}^\gamma}{\gamma!}u^{2g-2}v^\beta
\]
with
\[
\frac{{\bf x}^\gamma}{\gamma!}:=\prod_{l,k}\frac{x_{l,k}^{m_{l,k}(\gamma)}}{m_{l,k}(\gamma)!}.
\]
If $\cX$ is toric, then $\Mbar_{g,\gamma}(\cX,\beta)$ is not proper, but the GW partition function can still be defined by choosing a CY $\bC^*$ action on $\cX$ and replacing the above integral with
\[
\sum_{F\subset \Mbar_{g,\gamma}(\cX,\beta)^{\bC^*}}\int_{[F]^{vir}}\frac{1}{e^{eq}\left(\cN_F\right)}
\]
where the sum is over the fixed loci and the denominator is the equivariant Euler class of the normal bundle.

\subsubsection{Donaldson-Thomas Theory}

DT invariants are intersection numbers on $Hilb(\cX)$, the Hilbert scheme of substacks in $\cX$.  $Hilb(\cX)$ decomposes into subschemes indexed by the (compactly supported) $K$ group of coherent sheaves.  Let $\cO_{l,k}$ denote the skyscraper sheaf supported on a generic point of the orbifold line $l$ for which $\bZ_{n(l)}$ acts by multiplication by $\exp\left( \frac{2\pi\sqrt{-1}}{n(l)}k \right)$.  For $\gamma$ and $\beta$ as above, let $Hilb_{\gamma}(\cX,\beta)$ denote the component of the Hilbert scheme indexed by the class $[\cO_\beta]+\sum_{l,k}m_{l,k}[\cO_{l,k}]\in K(\cX).$  In \cite{bcy:otv}, orbifold DT invariants are defined via Behrend's constructible function $\nu:Hilb_{\gamma}(\cX,\beta)\rightarrow \bZ$ \cite{b:dttivmg}.  More precisely, the (multi-regular) DT partition function is defined by
\[
DT({\bf q},v):=\sum_{\beta,\gamma}e(Hilb_{\gamma}(\cX,\beta),\nu){\bf q}^\gamma v^\beta
\]
where
\[
e(Hilb_{\gamma}(\cX,\beta),\nu):=\sum_{k\in\bZ}ke(\nu^{-1}(k))
\]
with $e(-)$ the topological Euler characteristic.  For our purposes, we will be most interested with the reduced partition function
\[
DT'({\bf q},v):=\frac{DT({\bf q},v)}{DT({\bf q},v=0)}
\]

Notice that $\sum_{k}[\cO_{l,k}]=[\cO_{pt}]$ where $pt$ is a generic (smooth) point on $\cX$.  For later convenience, we introduce an additional variable $q$ and the relations $\prod_k q_{l,k}=q$ for any $l$.

\subsubsection{The GW/DT Correspondence}\label{sec:conjecture}

In order to make Conjecture \ref{gwdt} precise we define the change of variables ${\bf q}({\bf x},u)$ by
\[
 q\leftrightarrow -e^{\sqrt{-1}u}, \hspace{.5cm} q_{l,k}\leftrightarrow\xi_{n(l)}^{-1}e^{-\sum_i\frac{\xi_{n(l)}^{-ik}}{n(l)}(\xi_{2n(l)}^i-\xi_{2n(l)}^{-i})x_{l,i}} \hspace{.5cm} (k>0),
\]

With this change of variables, it is immediate that conjecture \ref{gwdt} generalizes the smooth GW/DT correspondence proposed by Maulik-Nekrasov-Okounkov-Pandharipande \cite{mnop:gwdti}.  Moreover, Maulik-Oblomkov-Okounkov-Pandharipande proved the smooth GW/DT correspondence in the toric case \cite{moop:gwdtc}, and Pixton-Pandharipande proved the correspondence for complete intersections in products of projective spaces \cite{pp:gwpcftq}.  In the orbifold setting, the authors have proved that Conjecture \ref{gwdt} is true for local cyclic orbifold lines \cite{z:gmvf,rz:ggmv}.

\subsection{Plan of the Paper}

In Section \ref{moduli} we set up the necessary notation to make precise the local GW and DT partition functions, ie. the framed orbifold vertex, and we state the local correspondence.  In Section \ref{sec:relations} we localize auxilary integrals into certain toric orbifold surfaces to develop relations which determine the two-leg GW orbifold vertex from an initial case.  We prove that the initial GW and DT two-leg vertices agree in the inital case in Sections \ref{sec:initial} and \ref{sec:comb}.  We prove that the local correspondence implies the global correspondence in Section \ref{sec:gluing}.

\subsection{Acknowledgements}

The authors are grateful to J. Bryan and B. Young for enlightening conversation.  They are particularly indebted to R. Cavalieri and C.-C. Liu for their expert advice. D. R. has been supported by NSF RTG grants DMS-0943832 and DMS-1045119.

\section{Definitions and Notations} \label{moduli}
In this section we set up notation which will be used throughout the paper and we give a precise statement of the main results.

\subsection{Partitions}
Our formulas are indexed by two types of combinatorial objects: ordinary partitions and $n$-partitions. Concretely, a partition $\tau$ is a sequence of positive integers $(\tau_1\geq\cdots\geq \tau_{l(\tau)})$.  We write $|\tau|:=\tau_1+\cdots+\tau_{l(\tau)}$ for the \emph{size} of $\tau$ and $l(\tau)$ for the \emph{length}. It is well known that partitions of size $d$ naturally index conjugacy classes and irreducible representations of the symmetric group $S_d$, see eg. \cite{m:sfhp}.  If $\tau$ and $\rho$ are partitions, we let $\chi_\rho(\tau)$ denote the value of the character of the irreducible representation $\rho$ on the conjugacy class $\tau$. Let $\tau'$ denote the transpose of $\tau$.  We also define
\begin{equation}\label{orderofcent1}
z_\tau:=|\Aut(\tau)|\prod \tau_i
\end{equation}
to be the order of the centralizer of any element in the conjugacy class $\tau$.

Fix a positive integer $n$ and a generator of the cyclic group
\[
\bZ_n=\left\langle\xi_n:=e^{\frac{2\pi\sqrt{-1}}{n}}\right\rangle.
\]
When no confusion arises, we write the generator simply as $\xi$. An $n$-partition is an ordered $n$-tuple of ordinary partitions:
$$
\mu=\left((\mu_1^0 \geq\cdots\geq  \mu_{l_0(\mu)}^0),...,(\mu_{1}^{n-1} \geq\cdots\geq  \mu_{l_{n-1}(\mu)}^{n-1})\right)
$$
with $\mu_j^i\in\bN$.  Let $\mu^i=(\mu_{1}^i,...,\mu_{l_i(\mu)}^i)$ denote the partition indexed by $i$ and let $\mu^{tw}$ correspond to the $n$-tuple of \emph{twisted} partitions $(\emptyset,\mu^1,...,\mu^{n-1})$.  Let $l(\mu):=\sum l_i(\mu)$ denote the \emph{length} of $\mu$ and $|\mu|:=\sum |\mu^i|$ the \emph{size}.  At times it will be convenient to write $\mu$ as a multiset $\{\xi^i\mu_{j}^i\}$ where the power of $\xi$ keeps track of which $\mu^i$ the $\mu_{j}^i$ came from.  Let $\underline\mu$ denote the underlying partition of $\mu$ that forgets the $\bZ_n$ decorations.  We define $-\mu:=\{\xi^{-i}\mu_j^i\}$, i.e. it is the $n$-partition with opposite twistings.

Similar to ordinary partitions, $n$-partitions of size $d$ naturally index conjugacy classes and irreducible representations of $\bZ_n\wr S_d$, see eg. \cite{m:sfhp}.  If $\mu$ and $\lambda$ are $n$-partitions, we let $\chi_{\lambda}(\mu)$ denote the value of the character of the irreducible representation $\lambda$ on the conjugacy class $\mu$. We also define
\begin{equation}\label{orderofcent2}
z_{\mu}:=|\Aut(\mu)|\prod n\mu_{j}^i
\end{equation}
to be the order of the centralizer of any element in the conjugacy class $\mu$.

Suppose $\lambda=(\lambda^0,...,\lambda^{n-1})$.  Via $n$-quotients (see \cite{rz:ggmv} Section 6) $\lambda$ can be identified with a partition of $n|\lambda|$.  We denote this corresponding partition by $\bar\lambda$ and we write $s_\lambda$ for the loop Schur function corresponding to $\bar\lambda$ (cf. Appendix \ref{sec:loop}).

\begin{notation}In order to distinguish our different combinatorial indices throughout the paper, we exclusively reserve $\tau$, $\eta$, $\rho$, and $\omega$ to denote ordinary partitions.  Moreover, $\tau$ and $\eta$ are used to index conjugacy classes of $S_d$ whereas $\rho$ and $\omega$ are used for representations.  Similarly, we reserve $\mu$, $\nu$, $\lambda$, and $\sigma$ for $n$-partitions.  We use $\mu$ and $\nu$ to index conjugacy classes of $\bZ_n\wr S_d$ while we save $\lambda$ and $\sigma$ for representations.
\end{notation}

\subsection{Gromov-Witten Theory: The Local Picture}
In this section, we recall the definition of the $A_n$ Gromov-Witten vertex and we define a slight modification which will be useful in our formulas.  Let us first recall some the relevant moduli spaces and tautological classes which appear in the vertex formulas.

Let $\gamma$ be a tuple of elements in $\bZ_n$ and let $m_i(\gamma)$ be the number of occurrences of $\xi^i\in\bZ_n$ in $\gamma$. Let $\Mbar_{g,\gamma}(\cB\bZ_n)$ denote the moduli stack of stable maps to the classifying space with $m_i(\gamma)$ marked points twisted by $\xi^i$.  By the definition of $\cB\bZ_n$, $\Mbar_{g,\gamma}(\cB\bZ_n)$ parametrizes degree $n$ covers of the source curve, ramified over the twisted points, with an action of $\bZ_n$ which exhibits the source curve as a quotient of the cover.  Let
$$
p:\mathcal{U}_h\rightarrow \Mbar_{g,\gamma}(\cB\bZ_n)
$$
be the universal covering curve of genus $h$ where $h$ is computed via the Riemann-Hurwitz formula.  The Hodge bundle on $\Mbar_{g,\gamma}(\cB\bZ_n)$ is the rank $h$ bundle defined by
$$
\mathbb{E}:=p_*\omega_h
$$
where $\omega_h$ is the relative dualizing sheaf of $p$.  $\bZ_n$ naturally acts on $\mathbb{E}$ and its dual $\mathbb{E}^\vee$.  For any $\zeta\in\bZ_n$, we define $\mathbb{E}_\zeta$ and $\mathbb{E}_\zeta^\vee$ to be the $\zeta$-eigenbundles of $\mathbb{E}$ and $\mathbb{E}^\vee$, respectively.  They are related by the formula $(\mathbb{E}_\zeta)^\vee=\mathbb{E}_{\zeta^{-1}}^\vee$.  We also have the formula
$$
\mathbb{E}_{\zeta^{-1}}^\vee=R^1\pi_*f^*\cO_{\zeta}
$$
where $\pi$ is the map from the universal source curve, $f$ is the universal map, and $\cO_{\zeta}$ is the line bundle with isotropy acting by multiplication by $\zeta$.  The lambda classes are defined as the chern classes of these bundles:
$$
\lambda_j^{\zeta}:=c_j\left(\mathbb{E}_\zeta\right)
$$

By forgetting the orbifold structure of the curve, there is a universal coarse curve
$$
q:\mathcal{U}_{g,|\gamma|}\rightarrow\Mbar_{g,\gamma}(\cB\bZ_n)
$$
along with a section $s_p$ for each marked point $p$.  We define the cotangent line bundles by
$$
\mathbb{L}_p:=s_p^*\omega_g
$$
where $\omega_g$ is the relative dualizing sheaf of $q$. The psi classes on $\Mbar_{g,\gamma}(\cB\bZ_n)$ are defined by
$$
\psi_p:=c_1\left(\mathbb{L}_p\right)
$$

Throughout the rest of the paper, we now restrict to the case where $\gamma$ consists only of \emph{nontrivial} elements in $\bZ_n$.  Let $\mu$ be an $n$-partition and let $\tau^+$ and $\tau^-$ be ordinary partitions and let $(\tau^+,\tau^-,\mu)$ denote the corresponding vector of $n+2$ partitions.  In order to write our formulas more concisely, we assign to each part $\kappa$ of $(\tau^+,\tau^-,\mu)$ a tuple $(l,m,k,d)\in\{1,2,3\}\times\{1,n\}\times\bZ/n\bZ\times\bN$ as follows:
\[
\Delta(\kappa)=\begin{cases}
(1,n,\tau_i^+ \text{ mod } n, \tau_i^+) &\text{ if }\kappa=\tau_i^+\in\tau^+\\
(2,n,-\tau_i^- \text{ mod }n, \tau_i^-) &\text{ if }\kappa=\tau_i^-\in\tau^-\\
(3,1,i,\mu_j^i) &\text{ if }\kappa=d_j^i\in\mu\\
\end{cases}
\]
We denote by $k(\tau^+,\tau^-,\mu)\in\bZ_n^{l(\tau^+)+l(\tau^-)+l(\mu)}$ the vector obtained by post composing $\Delta$ with the map $(l,m,k,d)\rightarrow\xi^k$ and we let $k_0(\tau^+,\tau^-,\mu)$ denote the number of ones in $k(\tau^+,\tau^-,\mu)$.  We will also use the notation $k_0^-(\tau^+)$, $k_0^+(\tau^-)$, and $k_0(\mu)$ for the obvious restrictions of $k_0$.

Let $w_1$, $w_2$, and $w_3$ be parameters satisfying $w_1+w_2+w_3=0$ and let $(r_1,r_2,r_3)=(1,-1,0)$ be the weights of the $\bZ_n$ representation giving the $3$-fold $A_{n-1}$ singularity.  By convention we define $w_4:=w_1$, $w_5:=w_2$, and similar for the $r_i$.  Following \cite{r:lgoa}, the particular Hodge integrals we are interested in take the form
\begin{align}
\nonumber V_{g,\gamma}&(\tau^+,\tau^-,\mu;w):=\frac{w_3^{l(\tau^+)+l(\tau^-)+l(\mu)-1 }(w_1w_2)^{k_0(\tau^+,\tau^-,\mu)}\prod_\kappa D(\Delta(\kappa))}{|\Aut(\tau^+)||\Aut(\tau^-)||\Aut(\mu)|}\\
&\cdot \int_{\Mbar_{g,\gamma+k(\tau^+,\tau^-,\mu)}(\cB\bZ_n)}\frac{\Lambda^{1}(w_1)\Lambda^{-1}(w_2)\Lambda^0(w_3)}{\delta(w)\prod_i\left(\frac{nw_1}{d_i^+}-\psi_i^+\right)\prod_i\left(\frac{nw_2}{d_i^-}-\psi_i^-\right)\prod_{i,j}\left(\frac{w_3}{d_j^i}-\psi_{i,j}\right)}
\end{align}
where $D$ is the (positively oriented) disk function
\[
D(l,m,k,d):=\left(\frac{d}{mw_l}\right)^{\delta_{0,k}}\frac{m}{d\left\lfloor\frac{d}{m}\right\rfloor!}\frac{\Gamma\left( \frac{dw_{l+1}}{mw_l}+\left\langle\frac{-kr_{l+2}}{n}\right\rangle+\frac{d}{m} \right)}{\Gamma\left( \frac{dw_{l+1}}{mw_l}-\left\langle\frac{-kr_{l+1}}{n}\right\rangle+1 \right)},
\]
the terms $\Lambda^i(t)$ package Hodge classes
\[
\Lambda^i(t):=(-1)^{\rk\left(\bE_{\xi^i}\right)}\sum_{j=0}^{\rk\left(\bE_{\xi^i}\right)}(-t)^{\rk\left(\bE_{\xi^i}\right)-j}\lambda_j^{\xi^i},
\]
 and $\delta(w)$ is the function which takes value $w_1w_2$ on the connected component parametrizing trivial covers of the source curve and takes value $1$ on all other components.

\begin{remark}
To obtain the negatively oriented disk functions, simply transpose the indices $i+1$ and $i+2$. Ultimately, changing the orientation merely introduces a sign and only becomes important in the orbifold vertex gluing algorithm.
\end{remark}

\begin{remark}
A bit of care is needed to define the exceptional cases where the moduli space in $V$ is undefined.  We adopt the following standard conventions.
\[
\int_{\Mbar_{0,(1)}(\cB\bZ_n)}\frac{1}{\delta(w)(a-\psi_1)}:=\frac{a}{w_1w_2n},
\]
\[
\int_{\Mbar_{0,(\xi^i,\xi^j)}(\cB\bZ_n)}\frac{1}{\delta(w)(a-\psi_1)}:=\begin{cases}
\frac{1}{n} & i+j=n\\
0 & \text{ else, }
\end{cases}
\]
and
\[
\int_{\Mbar_{0,(\xi^i,\xi^j)}(\cB\bZ_n)}\frac{1}{\delta(w)(a_1-\psi_1)(a_2-\psi_2)}:=\begin{cases}
\frac{1}{(w_1w_2)^{\delta_{0,i}}n(a_1+a_2)} & i+j = 0 \text{ mod } n\\
0 & \text{ else. }
\end{cases}
\]
Note that in the second case, we will never have $j=0$ due to the condition that $\gamma$ has only nontrivial elements of $\bZ_n$.
\end{remark}

\begin{remark}
It is not hard to check that $V_{g,\gamma}(\tau^+,\tau^-,\mu;w)$ is homogeneous of degree $0$ in the parameters $w_i$ and is therefore invariant under simultaneous rescaling.
\end{remark}

Introduce formal variables $u$ and $x_i$ to track genus and marks.  We define
\begin{equation}\label{def:vertex1}
V_{\tau^+,\tau^-,\mu}(x,u;w):=\sum_{g,\gamma}V_{g,\gamma}(\tau^+,\tau^-,\mu;w)u^{2g-2+l}\frac{x^\gamma}{\gamma!}
\end{equation}
where
\[
l:=l(\tau^+)+l(\tau^-)+l(\mu).
\]

Also introduce the variables $p_\tau^+$, $p_\tau^-$, $p_\mu$ with formal multiplication defined by concatenating indexing partitions whenever the superscripts agree.  We denote the disconnected vertex by
\begin{equation}\label{def:vertex2}
V_{\tau^+,\tau^-,\mu}^\bullet(x,u;w):=\exp\left( \sum_{\eta^+,\eta^-,\nu }V_{\eta^+,\eta^-,\nu }(x,u;w)p_{\eta^+}p_{\eta^-}^+ p_{\nu}^- \right)\left[ p_{\tau^+}^-p_{\tau^-}^+ p_{\mu}\right]
\end{equation}
where  $[-]$ denotes ``the coefficient of''.  By definition, $V_{\tau^+,\tau^-,\mu}^\bullet(x,u;w)$ is the GW $A_{n-1}$ vertex defined in \cite{r:lgoa}.  For our current purposes, it is more convenient to work with a slight modification.

\begin{definition}\label{framedvertex}
The \emph{framed GW $A_{n-1}$ vertex} is defined by
\begin{align}
\nonumber\tilde{V}_{\tau^+,\tau^-,\mu}^{\bullet,\alpha}(w):=&\alpha_+^{l(\tau^+)+|\tau^+|}\alpha_-^{l(\tau^-)+|\tau_-|}(-1)^{|\mu|+\sum_i\left(\left\lfloor-\frac{\tau^-_i}{n}\right\rfloor\right)}\\
&\cdot\sqrt{-1}^{-l(\tau^+)-l(\tau^-)}\prod_{i=0}^{n-1}(\sqrt{-1}\xi_{2n}^i)^{l_i(\mu)}V_{\tau^+,\tau^-,\mu}^\bullet(x,u;w)\label{framedv}
\end{align}
where $\alpha=(\alpha_+,\alpha_-)\in\{1,-1\}^2$ will be required for the gluing algorithm.  We similarly define the \emph{connected} series $\tilde V_{g,\gamma}^\alpha(\tau^+,\tau^-,\mu;w)$ and $\tilde V_{\tau^+,\tau^-,\mu}^\alpha(x,u;w)$ by multiplying by the same prefactor that appears in \eqref{framedv} so that equations \eqref{def:vertex1} and \eqref{def:vertex2} continue to hold with $V$ replaced by $\tilde{V}$.
\end{definition}

\begin{remark}\label{gwspec}
It follows from the definitions that $\tilde V^{\bullet}_{\emptyset,\emptyset,\mu}(a,-a-1,1) = \tilde V^\bullet_\mu(a)$ where the latter is defined in \cite{rz:ggmv}.
\end{remark}

\subsection{Donaldson-Thomas Theory: The Local Picture}

Let $\bq=(q_0,\dots,q_{n-1})$ be formal variables with indices computed modulo $n$.  Define the variables $\fq_i$ recursively by $\fq_0:=1$ and
\[
\fq_t:=q_t\fq_{t-1}
\]
so that $\{\dots,\fq_{-2},\fq_{-1},\fq_0,\fq_1,\fq_2,\dots\}=\{\dots,q_{-2}^{-1}q_{-1}^{-1},q_{-1}^{-1},1,q_1,q_1q_2,\dots\}$.

For a colored Young diagram $\bar\lambda$ corresponding to a $n$-partition $\lambda$ via $n$-quotients, denote the sizes of the rows in $\bar\lambda$ by $(\bar\lambda_1,\bar\lambda_2,\dots)$.  Define the variables $\fq_{\bullet-\lambda}$ by
\[
\fq_{\bullet-\lambda}:=\{\fq_{-\bar\lambda_1}, \fq_{1-\bar\lambda_2}, \fq_{2-\bar\lambda_3},\dots\}.
\]
In particular, we denote $\fq_{\bullet-\emptyset}=\fq_\bullet$.  An overline on an expression in the $q$ variables denotes interchanging $q_i\leftrightarrow q_{-i}$.  We use $\fq_{\bullet-\lambda'}$ to denote the variables obtained in this way from $\bar\lambda'$.

In \cite{bcy:otv}, an explicit combinatorial formula for the DT $A_n$ vertex is computed.  Motivated by their computations and Lemma \ref{hookcontent}, we define
\begin{align}\label{eq:divert}
P_{\rho^+,\rho^-,\lambda}:=s_{\lambda}(\bq)\sum_\omega q_0^{-|\omega|}\overline{s_{\rho^+/\omega}(\fq_{\bullet-\lambda})}s_{\rho^-/\omega}(\fq_{\bullet-\lambda'})
\end{align}
where $s_{\lambda}(\bq)$ denotes the loop Schur function of $\bar\lambda$ in the variables $(q_0,\dots,q_{n-1})$ and $s_{\rho/\omega}$ denotes a skew Schur function.

We modify $P$ to incorporate the framing.

\begin{definition}\label{def:dt}
The \emph{framed DT $A_n$ vertex} is defined by
\begin{align*}
\tilde P_{\rho^+,\rho^-,\lambda}(w):&=(-1)^{|\lambda|}q^{\frac{|\lambda|}{2}}\frac{\chi_{\bar\lambda}(n^{|\lambda|})}{\dim(\lambda)}\left( \left(-\xi_{2n}\right)^{-|\lambda|} \prod_k \xi_n^{-k|\lambda_k|} \right)^{\frac{nw_1}{w_3}}\\
&\cdot \left( -q^{\frac{1}{2}}q_1^{-\frac{1}{n}}\cdots q_{n-1}^{-\frac{n-1}{n}}\right)^{|\rho^+|}\left(-q^{\frac{1}{2}}q_1^{-\frac{n-1}{n}}\cdots q_{n-1}^{-\frac{1}{n}} \right)^{|\rho^-|}\\
&\cdot\left(\prod_{(i,j)\in\rho^+}q^{i-j} \right)^{\frac{w_3}{nw_1}}\left(\prod_{(i,j)\in\rho^-}q^{i-j} \right)^{\frac{w_3}{nw_2}}\left(\prod_{(i,j)\in\bar\lambda}q_{j-i}^{i-j} \right)^{\frac{w_1}{w_3}} P_{\rho^+,\rho^-,\lambda}
\end{align*}
and
\[
\tilde P^\alpha_{\rho^+,\rho^-,\lambda}(w)=\tilde P_{(\rho^+)^{\alpha_+},(\rho^-)^{\alpha_-},\lambda}(w)
\]
where $\rho^1:=\rho$ and $\rho^{-1}:=\rho'$.

\end{definition}

\begin{remark}
Notice that $\tilde P^\alpha_{\rho^+,\rho^-,\lambda}(w)$ can be defined for arbitrary parameters $w$ because we have chosen a branch of the logarithm in the first line.
\end{remark}

\begin{remark}\label{dtspec}
It follows from the definitions that $\tilde P_{\emptyset,\emptyset,\lambda}(a,-a-1,1)=\tilde P_\lambda(a)$ where the latter was defined in \cite{rz:ggmv}.
\end{remark}

\subsection{The Correspondence}

We claim that the framed vertex theories are equivalent after an identification of variables.

\begin{conjecture}\label{conjecture}
After the identification of variables
\[
q\leftrightarrow e^{\sqrt{-1}u}, \hspace{.5cm} q_k\leftrightarrow\xi_n^{-1}e^{-\sum_i\frac{\xi_n^{-ik}}{n}(\xi_{2n}^i-\xi_{2n}^{-i})x_i} \hspace{.5cm} (k>0),
\]
we have an identification of framed vertex theories:
\[
\tilde V_{\tau^+,\tau^-,\mu}^{\bullet,\alpha}(w)=\sum_{\rho^+,\rho^-,\lambda}\tilde P^\alpha_{\rho^+,\rho^-,\lambda}(w)\frac{\chi_{\rho^+}(\tau^+)}{z_{\tau^+}}\frac{\chi_{\rho^-}(\tau^-)}{z_{\tau^-}}\frac{\chi_\lambda(\mu)}{z_\mu}.
\]
\end{conjecture}

\begin{remark}
By the main result of \cite{rz:ggmv} and Remarks \ref{gwspec} and \ref{dtspec}, Conjecture \ref{conjecture} is true if $\tau^+=\tau^-=\emptyset$.
\end{remark}

By analyzing the gluing algorithm for the orbifold vertex, we show the following.

\begin{theorem}\label{thm:imply}
Conjecture \ref{conjecture} implies the GW/DT correspondence for toric CY $3$-folds with transverse $A_n$ singularities.
\end{theorem}

The main result of this paper is the two-leg correspondence.

\begin{theorem}\label{thm:main}
Conjecture \ref{conjecture} is true if one of $\tau^+$, $\tau^-$, or $\mu$ is empty.  As a result, the GW/DT correspondence holds for local toric surfaces when the total space has transverse $A_n$ singularities.
\end{theorem}

\subsection{Symmetry}

The generating functions $\tilde V$ and $\tilde P$ exhibit many symmetries which follow easily from the definitions.  In particular, if we let $\bar x$ denote the exchange of variables $x_i\leftrightarrow x_{n-i}$, $\bar\alpha=(\alpha_-,\alpha_+)$, and $\bar w=(w_2,w_1,w_3)$, then we have
\begin{equation}\label{gwsym2}
\tilde V^{\bullet,\alpha}_{\tau^+,\tau^-,\mu}(x,u;w)=(-1)^{|\mu|+l(\mu)}\prod_{i}\xi_n^{il_i(\mu)} \tilde V^{\bullet,\bar\alpha}_{\tau^-,\tau^+,-\mu}(\bar x,u;\bar w) .
\end{equation}
On the DT side, we have
\begin{equation}\label{dtsym}
\tilde P^\alpha_{\rho^+,\rho^-,\lambda}(q;w)= \tilde P^{\bar\alpha}_{\rho^-,\rho^+,\lambda'}(\bar q, \bar w)
\end{equation}

Moreover, $\chi_{\lambda'}(-\mu)=(-1)^{|\mu|+l(\mu)}\prod_{i}\xi_n^{-il_i(\mu)}\chi_\lambda(\mu)$ and the identification of variables is compatible with $x\rightarrow \bar x$, $q\rightarrow \bar q$.  This implies that conjecture \ref{conjecture} is invariant under simultaneously exchanging $\tau^+\leftrightarrow\tau^-$, $\rho^+\leftrightarrow\rho^-$, $\mu\rightarrow-\mu$,  $\alpha\rightarrow \bar\alpha$, $w\rightarrow\bar w$, $x\rightarrow \bar x$, and $q\rightarrow\bar q$.  In particular, this implies that the conjecture holds for $\tau^+=\emptyset$ if and only if it holds for $\tau^-=\emptyset$.

\section{Relations via Localization}\label{sec:relations}

In this section, we define and study auxiliary integrals on moduli spaces of relative stable maps into certain orbifold surfaces.  Computing the integrals via localization, we produce relations which completely determine the two-leg GW $A_{n-1}$ vertex from an initial value which we compute in Sections \ref{sec:initial} and \ref{sec:comb}.

\subsection{Targets}

\subsubsection{Symmetric Case}

Let $\Sigma_s$ be the two dimensional fan generated by the following five rays: $v_1=(n+1,-n),v_2=(-1,1),v_3=(0,1),v_4=(1,-1),v_5=(0,-1)$ and let $X_s$ be the toric stack defined by $\Sigma_s$. Let $D_i$ be the torus invariant divisor corresponding to $v_i$. Then $X_s$ has a unique singular point at $z_0=D_1\cap D_3$ with isotropy group $\bZ_n$.  We fix an isomorphism $\bZ_n\cong\left\langle \xi_n \right\rangle$ by requiring that the induced representation of $TD_1|_{z_0}$ is multiplication by $\xi_n^{-1}$.   See Figure \ref{fig:s} for the fan and polygon corresponding to $X_s$.

\begin{figure}
\begin{center}
\psfrag{D1}{$D_1$}
\psfrag{D2}{$D_2$}
\psfrag{D3}{$D_3$}
\psfrag{D4}{$D_4$}
\psfrag{D5}{$D_5$}
\psfrag{z0}{$z_0$}
\psfrag{z+}{$z_+$}
\psfrag{z_}{$z_-$}
\psfrag{tz+}{$\tilde{z}_+$}
\psfrag{tz_}{$\tilde{z}_-$}
\psfrag{w1}{$w_1$}
\psfrag{w2}{$w_2$}
\psfrag{-nw1}{$-nw_1$}
\psfrag{-w3}{$-w_3$}
\psfrag{-nw2}{$-nw_2$}
\psfrag{nw1}{$nw_1$}
\psfrag{w3}{$w_3$}
\psfrag{v1}{$v_1$}
\psfrag{v2}{$v_2$}
\psfrag{v3}{$v_3$}
\psfrag{v4}{$v_4$}
\psfrag{v5}{$v_5$}
\includegraphics{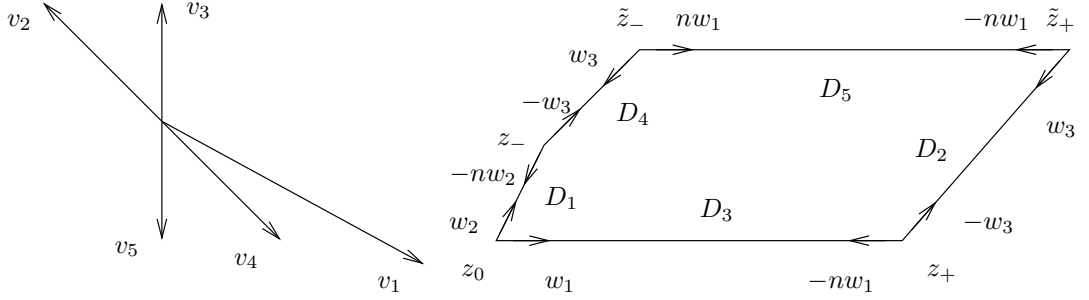}
\caption{The fan and the polytope of the coarse moduli space of $X_a$}\label{fig:s}
\end{center}
\end{figure}

\subsubsection{Asymmetric Case}

Let $\Sigma_a$ be the two dimensional fan generated by the following five rays: $v_1=(1,0),v_2=(-1,0),v_3=(0,1),v_4=(0,-1),v_5=(-1,-1)$ and let $\tilde X_a$ be the corresponding toric variety ($\tilde X_a$ is the blowup of $\bP^1\times\bP^1$ at a point).  Let $D_i$ be the torus invariant divisor corresponding to $v_i$ and let $X_a$ be obtained from $\tilde X_a$ by root construction of order $n$ along the divisor $D_3\subset X$.  The root construction replaces $D_3$ with a trivial $\bZ_n$ gerbe over $D_3$ and we identify $\bZ_n\cong\left\langle \xi_n \right\rangle$  by requiring that the induced representation on the normal bundle $\cN_{D_3/X_a}$ is multiplication by $\xi_n$.  We abuse notation and refer to the orbifold divisor also as $D_3$.  See Figure \ref{fig:a} for the fan and polygon corresponding to $X_a$.

\begin{figure}
\begin{center}
\psfrag{D1}{$D_1$}
\psfrag{D2}{$D_2$}
\psfrag{D3}{$D_3$}
\psfrag{D4}{$D_4$}
\psfrag{D5}{$D_5$}
\psfrag{z0}{$z_0$}
\psfrag{z+}{$z_+$}
\psfrag{z_}{$z_-$}
\psfrag{tz+}{$\tilde{z}_+$}
\psfrag{tz_}{$\tilde{z}_-$}
\psfrag{w1}{$w_1$}
\psfrag{w3}{$w_3$}
\psfrag{-w3}{$-w_3$}
\psfrag{-nw1}{$-nw_1$}
\psfrag{w3-nw1}{$w_3-nw_1$}
\psfrag{nw1-w3}{$nw_1-w_3$}
\psfrag{v1}{$v_1$}
\psfrag{v2}{$v_2$}
\psfrag{v3}{$v_3$}
\psfrag{v4}{$v_4$}
\psfrag{v5}{$v_5$}
\includegraphics[scale=0.8]{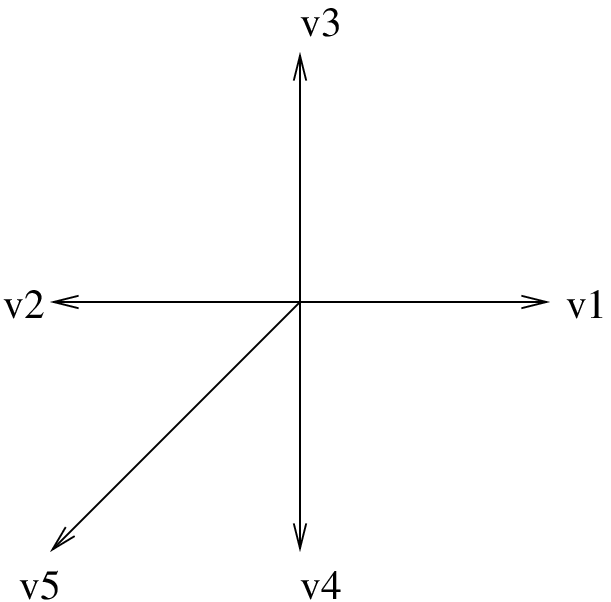}
\includegraphics[scale=0.8]{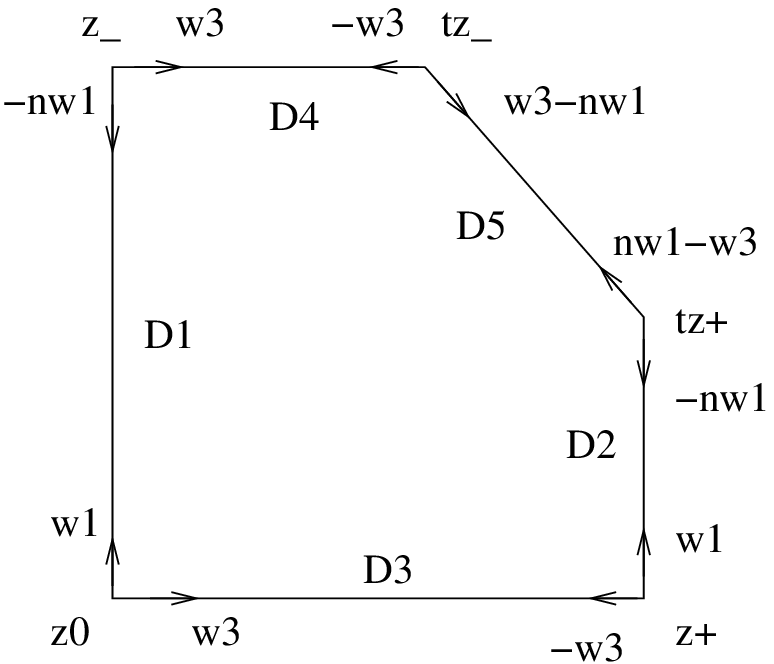}
\end{center}
\caption{The fan and the polytope of $\tilde{X}_a$}\label{fig:a}
\end{figure}

\subsection{Auxiliary Integrals}

\subsubsection{Symmetric Case}

Let $\Mbar^\bullet_{\chi,\gamma}(X_s,\tau^+,\tau^-)$ denote the moduli space of (possibly disconnected) relative stable maps into $X_s$ with relative conditions given by $\tau^+$ along $D_2$ and $\tau^-$ along $D_4$.  Let $\pi:\cU\rightarrow\Mbar^\bullet_{\chi,\gamma}(X_s,\tau^+,\tau^-)$ be the universal curve, $f:\cU\rightarrow\cT$ the map to the universal target, and $\hat f:\cU\rightarrow X_s$ the map which postcomposes $f$ with the natural contraction.  Let $\cD_+, \cD_-\subset\cU$ be the divisors corresponding to the relative marked points on the source curve.  In the symmetric case, the integrals we investigate are
\begin{equation}\tag{$I_s$}\label{eqn:ints}
\frac{1}{|\Aut(\tau^+)||\Aut(\tau^-)|}\int_{\Mbar^\bullet_{\chi,\gamma}(X_s,\tau^+,\tau^-)}e\left( V_s \right)
\end{equation}
where $V_s$ is the obstruction bundle
\[
V_s:=R^1\pi_* \left(\hat f^*\cO_{X_s}(-D_5)\otimes\cO_{\cU}(-\cD_+-\cD_-) \right).
\]

\subsubsection{Asymmetric Case}

Let $\Mbar^\bullet_{\chi,\gamma}(X_s,\tau,\mu)$ denote the moduli space of relative stable maps into $X_a$ with relative conditions given by $\tau$ along $D_2$ and $\mu$ along $D_4$.  We define the obstruction bundle $V_a$ in the asymmetric case exactly as in the symmetric case, except we define $\cD_+\subset \cU$ to be the divisor corresponding to only the relative points in $\mu$ with \emph{trivial} isotropy and we replace the twisting divisor $-D_5$ with $-D_1-D_3$.  The integrals we investigate in the asymmetric case are
\begin{equation}\tag{$I_a$}\label{eqn:inta}
\frac{1}{|\Aut(\tau)||\Aut(\mu)|}\int_{\Mbar^\bullet_{\chi,\gamma}(X_s,\tau,\mu)}e\left( V_a \right)
\end{equation}

\subsection{Torus Action}

In order to apply the localization formula to our integrals, we equip the moduli spaces and the integrands with a $\bC^*$ action.  The actions on the moduli spaces are defined by postcomposing the map with a $\bC^*$ action on the target.  A compatible action on the integrand is obtained through a lift of the $\bC^*$ action to $\cO(-D_5)$ in the symmetric case and $\cO(-D_1-D_3)$ in the asymmetric case.  These actions are determined by the weights at fixed points of the target.  For any choice of $w_1,w_2\in\frac{1}{n}\bZ$ and $w_3\in\bZ$ with $w_1+w_2+w_3=0$, we obtain a $\bC^*$ action with weights at the fixed points collected in the following tables (see also Figures \ref{fig:s} and \ref{fig:a}).

\subsubsection{Symmetric Case}

\[
\begin{array}{ccc}
      &        T_{X_s}            & \cO_{X_s}(-D_5) \\
  z_0 &   w_1,w_2       & w_3 \\
  z_+ & -nw_1,-w_3         &   w_3     \\
  z_- &   -nw_2, -w_3      &   w_3      \\
\tz_+ & w_3, -nw_1  &     0           \\
\tz_- & w_3, nw_1 &     0
\end{array}
\]

\subsubsection{Asymmetric Case}

\[
\begin{array}{ccc}
      &        T_{X_a}            & \cO_{X_a}(-D_1-D_3) \\
  z_0 &  w_1,w_3        &  w_2 \\
  z_+ &  -w_3,w_1         &   -w_1        \\
  z_- &   w_3, -nw_1      &   -w_3       \\
\tz_+ & -nw_1, nw_1-w_3 &     0           \\
\tz_- & -w_3, -nw_1+w_3 &     0
\end{array}
\]

\subsection{Localization Graphs}

The localization formula reduces the auxiliary integrals \eqref{eqn:ints} and \eqref{eqn:inta} to integrals over the $\bC^*$ fixed loci of the moduli space.  The fixed loci can be indexed by (possibly disconnected) graphs and the integral can be represented as a localization graph sum, see \cite{l:ligwtaogwt} for the basics of localization in orbifold GW theory.  In both the symmetric and asymmetric cases, we have the following vanishing result on many of the contributions which significantly simplifies the graph sum.

\begin{lemma}\label{lem:vanish}
Let $\beta_i$ denote the curve class represented by $D_i$, then the auxiliary integrals \eqref{eqn:ints} and \eqref{eqn:inta} are zero for all curve classes $\sum
a_i\beta_i$ unless $a_2=a_4=a_5=0$.
\end{lemma}

\begin{proof}
The proof is analogous to the proof of Lemma 7.2 in \cite{llz:ftphi}.  The key point is that the $0$ weights of the fibers over $\tz_+$ and $\tz_-$ annihilate the contribution of the integral over any fixed locus which maps a rational curve to $D_5$.
\end{proof}

Due to the vanishing of Lemma \ref{lem:vanish}, the graphs contributing to the auxiliary integrals take on a simple form.  In particular, the fixed loci can be indexed by a tripartite graph as follows.

\subsubsection{Symmetric Case}

\begin{itemize}
\item Vertices correspond to connected components of $\hat f^{-1}(z_0)$, $\hat f^{-1}(z_-)$, and, $\hat f^{-1}(z_+)$, we group these into three sets $V_0$, $V_-$, and $V_+$, respectively.  Each vertex $v$ is labeled with a nonnegative integer $g_v$ to denote the genus of the connected component.
\item Edges correspond to rational curves connecting the contracted components, we group these into two sets (by Lemma \ref{lem:vanish}) $E_{0,-}$ and $E_{0,+}$.  Each edge $e$ is labeled with a positive integer $d_e$ which denotes the degree of the corresponding map between rational curves.  This labeling induces a partition $\eta_v$ for each $v\in V_-\cup V_+$ and it induces a pair of partitions $(\eta_v^-,\eta_v^+)$ for $v\in V_0$.
\item Each vertex $v\in V_-\cup V_+$ is labeled with an additional partition $\tau_v$ such that $|\tau_v|=|\eta_v|$ and $\bigcup_{v\in V_{\pm}}\tau_v=\tau^{\pm}$.
\end{itemize}

\subsubsection{Asymmetric Case}

In the asymmetric case the graphs have a few extra decorations arising from the orbifold structure on $D_3$.
\begin{itemize}
\item Each vertex $v\in V_+\cup V_0$ is labeled with a tuple $\gamma_v$ of nontrivial elements in $\bZ_n$ corresponding to the orbifold structure on the contracted component.
\item Instead of labeling each edge $e\in E_{0,+}$ with an integer, we label it with a complex number of the form $\xi^{k_e}d_e$.  This induces $n$-partitions $\nu_v$ for each $v\in V_0$ and $-\nu_v$ for $v\in V_+$.
\item Instead of labeling each $v\in V_+$ with ordinary partitions we label them with $n$-partitions $\mu_v$ such that $|\mu_v|=|\nu_v|$ and $\bigcup_{v\in V_+}\mu_v=\mu$.
\end{itemize}

\subsection{Graph Contributions}

We now collect the localization contribution from each vertex and write down the auxiliary integrals explicitly as graph sums.  For more details on the computation of the vertex contributions see \cite{z:gmvf} and \cite{rz:ggmv}.

\subsubsection{Symmetric Case}

To a vertex $v\in V_0$ we assign the contribution
\[
\Cont(v):=(-1)^{\sum_i\left\lfloor -\frac{(\eta_v^-)_i}{n} \right\rfloor}V_{g_v,\gamma_v}(\eta_v^+,\eta_v^-,\emptyset;w).
\]
To a vertex $v\in V_+$ we assign the contribution
\begin{align*}
\Cont(v):=&\frac{(-1)^{g_v-1+l(\tau_v)}}{|\Aut(\eta_v)|}\left( \frac{w_3}{nw_1} \right)^{2g_v-2+l(\eta_v)+l(\tau_v)}\\
&\cdot\prod (\eta_v)_i\int_{\Mbar_{g_v}(\bP^1;\eta_v,\tau_v)//\bC^*}\psi_0^{2g_v-3+l(\eta_v)+l(\tau_v)}.
\end{align*}
To a vertex $v\in V_-$ we assign the contribution
\begin{align*}
\Cont(v):=&\frac{(-1)^{g_v-1+l(\tau_v)}}{|\Aut(\eta_v)|}\left( \frac{w_3}{nw_2} \right)^{2g_v-2+l(\eta_v)+l(\tau_v)}\\
&\cdot\prod(\eta_v)_i\int_{\Mbar_{g_v}(\bP^1;\eta_v,\tau_v)//\bC^*}\psi_0^{2g_v-3+l(\eta_v)+l(\tau_v)}.
\end{align*}

By the localization formula, we can write the integral \eqref{eqn:ints} as a graph sum:
\begin{equation}\label{localize:s1}
\eqref{eqn:ints}=\frac{1}{|\Aut(\tau^+)||\Aut(\tau^-)|}\sum_\Gamma\frac{1}{|\Aut(\Gamma)|}\prod_v \Cont(v).
\end{equation}

In particular, if we set $w_s:=(\frac{1}{n},-\frac{1}{n},0)$, the contributions from vertices in $V_+$ and $V_-$ vanish and the integral evaluates to
\begin{equation}\label{localize:s2}
\eqref{eqn:ints}=(-1)^{\sum_i\left(\left\lfloor-\frac{(\tau_v^-)_i}{n}\right\rfloor\right)}V_{g,\gamma}(\tau^+,\tau^-,\emptyset;w_s)
\end{equation}

Defining rubber integral generating functions as in Appendix \ref{sec:rubber}, \eqref{localize:s1} and \eqref{localize:s2} together imply the following relation between the arbitrary symmetric two leg vertex and the one with specific weights $w_s$.
\begin{equation}\label{frame:s1}
\tilde V^{\bullet,\alpha}_{\tau^+,\tau^-,\emptyset}(w_s)=\sum_{\eta^+,\eta^-}\tilde V^{\bullet,\alpha}_{\eta^+,\eta^-,\emptyset}(w)z_{\eta^+}\tilde{H}^\bullet_{\eta^+,\tau^+}\left(\alpha_+\frac{w_3}{nw_1}\right)z_{\eta^-}\tilde{H}^\bullet_{\eta^-,\tau^-}\left(\alpha_-\frac{w_3}{nw_2}\right).
\end{equation}

\subsubsection{Asymmetric Case}

To a vertex $v\in V_0$ we assign the contribution
\[
\Cont(v):=(-1)^{|\nu|+l_0(\nu)}V_{g_v,\gamma_v}(\eta_v,\emptyset,\nu_v;w).
\]
To a vertex $v\in V_+$ we assign the contribution
\begin{align*}
\Cont(v):=&\frac{(-1)^{l_0(\nu_v)+g_v-1+\sum_{i\neq 0}\frac{n-i}{n}(m_i(\gamma_v)+l_i(\mu_v)+l_{n-i}(\nu_v))}}{|\Aut(\nu_v)|}\left(-\frac{w_1}{w_3}\right)^{2g_v-2+|\gamma_v|+l(\mu_v)+l(\nu_v)}\\
&\hspace{1cm}\cdot\left( \prod_{i}n\nu_i \right)\int_{\Mbar_{g_v,\gamma_v}(\bP^1\times\cB\bZ_n;-\nu_v,\mu_v)//\bC^*}\psi_0^{2g_v-3+|\gamma_v|+l(\nu_v)+l(\mu_v)}.
\end{align*}
To a vertex $v\in V_-$ we assign the contribution
\begin{align*}
\Cont(v):=&\frac{(-1)^{g_v-1+l(\tau_v)}}{|\Aut(\eta_v)|}\left( \frac{w_3}{nw_1} \right)^{2g_v-2+l(\eta_v)+l(\tau_v)}\\
&\cdot\prod(\eta_v)_i\int_{\Mbar_{g_v}(\bP^1;\eta_v,\tau_v)//\bC^*}\psi_0^{2g_v-3+l(\eta_v)+l(\tau_v)}.
\end{align*}

By the localization formula, we can write the integral \eqref{eqn:inta} as a graph sum:
\begin{equation}\label{localize:a1}
\eqref{eqn:inta}=\frac{1}{|\Aut(\mu)||\Aut(\tau)|}\sum_\Gamma\frac{1}{|\Aut(\Gamma)|}\prod_v \Cont(v).
\end{equation}

As our initial condition in the asymmetric case, we take $w_a:=(\frac{1}{n},-1-\frac{1}{n},1)$.  Then \eqref{localize:a1} implies the following relation between the arbitrary vertex and our initial condition.
\begin{equation}\label{frame:a1}
\sum_{\eta,\nu}\tilde V^{\bullet,\alpha}_{\eta,\emptyset,\nu}(w_a)z_{\eta}\tilde{H}^\bullet_{\eta,\tau}\left(\alpha_+\right)z_{\nu}\tilde{H}^\bullet_{\nu,\mu}\left(\frac{1}{n}\right)=\sum_{\eta,\nu}\tilde V^{\bullet,\alpha}_{\eta,\emptyset,\nu}(w)z_{\eta}\tilde{H}^\bullet_{\eta,\tau}\left(\alpha_+\frac{w_3}{nw_1}\right)z_{\nu}\tilde{H}^\bullet_{\nu,\mu}\left(\frac{w_1}{w_3}\right).
\end{equation}

\subsection{Framing Dependence}

We now compute explicitly the dependence of the GW $A_n$ vertex on the weights $w$.

\subsubsection{Symmetric Case}

Using the formulas from Appendix \ref{sec:rubber}, we can invert the formula \eqref{frame:s1} to obtain the following.

\begin{lemma}\label{frame:s2}
The framing dependence of the symmetric two leg vertex in the conjugacy class basis is
\[
\tilde V^{\bullet,\alpha}_{\tau^+,\tau^-,\emptyset}(w)=\sum_{\eta^+,\eta^-}\tilde V^{\bullet,\alpha}_{\eta^+,\eta^-,\emptyset}(w_s)z_{\eta^+}\tilde{H}^\bullet_{\eta^+,\tau^+}\left(-\alpha_+\frac{w_3}{nw_1}\right)z_{\eta^-}\tilde{H}^\bullet_{\eta^-,\tau^-}\left(-\alpha_-\frac{w_3}{nw_2}\right).
\]
\end{lemma}

If we define
\[
\hat P^\alpha_{\rho^+,\rho^-,\emptyset}(w):=\sum_{\tau^+,\tau^-}\tilde V^{\bullet,\alpha}_{\tau^+,\tau^-,\emptyset}(w)\chi_{\rho^+}(\tau^+)\chi_{\rho^-}(\tau^-)
\]
or equivalently
\[
V^{\bullet,\alpha}_{\tau^+,\tau^-,\emptyset}(w)=:\sum_{\rho^+,\rho^-}\hat P^\alpha_{\rho^+,\rho^-,\emptyset}(w)\frac{\chi_{\rho^+}(\tau^+)}{z_{\tau^+}}\frac{\chi_{\rho^-}(\tau^-)}{z_{\tau^-}},
\]
then Lemma \ref{frame:s2} is equivalent to the following.
\begin{lemma}\label{frame:s3}
The framing dependence of the symmetric two leg vertex in the representation basis is
\[
\hat P^\alpha_{\rho^+,\rho^-,\emptyset}(w)=e^{-\left(\alpha_+f_T(\rho^+)\sqrt{-1}u\frac{w_3}{nw_1}+\alpha_-f_T(\rho^-)\sqrt{-1}u\frac{w_3}{nw_2}\right)}\hat P^\alpha_{\rho^+,\rho^-,\emptyset}(w_s)
\]
\end{lemma}

After the prescribed change of variables (applying Lemma \ref{lem:central}), the prefactor in the representation basis becomes
\[
\left(\prod_{(i,j)\in\rho^+}q^{i-j} \right)^{\alpha_+\frac{w_3}{nw_1}}\left(\prod_{(i,j)\in\rho^-}q^{i-j} \right)^{\alpha_-\frac{w_3}{nw_2}}=\left(\prod_{(i,j)\in(\rho^+)^{\alpha_+}}q^{i-j} \right)^{\frac{w_3}{nw_1}}\left(\prod_{(i,j)\in(\rho^-)^{\alpha_-}}q^{i-j} \right)^{\frac{w_3}{nw_2}}.
\]
Comparing this to the framing dependence in Definition \ref{def:dt}, we see that in order to prove the symmetric correspondence it is enough to prove it for the case $w=w_s$.

\subsubsection{Asymmetric Case}

Using the formulas from Appendix \ref{sec:rubber}, we can rewrite the formula \eqref{frame:a1} in the following way.

\begin{lemma}\label{frame:a2}
The framing dependence of the asymmetric two leg vertex in the conjugacy class basis is
\[
\tilde V^{\bullet,\alpha}_{\tau,\emptyset,\mu}(w)=\sum_{\eta,\nu}\tilde V^{\bullet,\alpha}_{\eta,\emptyset,\nu}(w_a)z_\eta \tilde{H}^\bullet_{\eta,\tau}\left(\alpha_+\left(1-\frac{w_3}{nw_1}\right)\right)z_\nu \tilde{H}^\bullet_{\nu,\mu}\left(\frac{1}{n}-\frac{w_1}{w_3}\right)
\]
\end{lemma}

In particular, if we define
\[
\hat P^\alpha_{\rho,\emptyset,\lambda}(w):=\sum_{\mu,\tau}\tilde V^{\bullet,\alpha}_{\tau,\emptyset,\mu}(w)\chi_{\lambda}(\mu)\chi_{\rho}(\tau)
\]
or equivalently
\[
V^{\bullet,\alpha}_{\tau,\emptyset,\mu}(w)=:\sum_{\lambda,\rho}\hat P^\alpha_{\rho,\emptyset,\lambda}(w)\frac{\chi_{\lambda}(\mu)}{z_{\mu}}\frac{\chi_{\rho}(\tau)}{z_{\tau}},
\]
then Lemma \ref{frame:a2} is equivalent to the following.

\begin{lemma}\label{frame:a3}
The framing dependence of the asymmetric two leg vertex in the representation basis
\[
\hat P^\alpha_{\rho,\emptyset,\lambda}(w)=e^{\alpha_-f_T(\rho)\sqrt{-1}u\left(1-\frac{w_3}{nw_1}\right)+\left(f_T(\lambda)\sqrt{-1}u+\sum_i f_i(\lambda)x_i\right)\left(\frac{1}{n}-\frac{w_1}{w_3}\right)}\hat P^\alpha_{\rho,\emptyset,\lambda}(w_a)
\]
\end{lemma}

After the prescribed identification of variables (again applying Lemma \ref{lem:central}), the prefactor in the representation basis becomes
\[
\left(\prod_{(i,j)\in\rho^{\alpha_-}}q^{i-j} \right)^{\frac{w_3}{nw_1}-1}\left( \left(-\xi_{2n}\right)^{-|\lambda|} \prod_k \xi_n^{-k|\lambda_k|} \prod_{(i,j)\in\bar\lambda}q_{j-i}^{\frac{i-j}{n}} \right)^{\frac{nw_1}{w_3}-1}
\]
which is compatible with the framing dependence in Definition \ref{def:dt}. We have thus reduced the asymmetric correspondence to the case of $w=w_a$.

\section{Computation of Initial Values}\label{sec:initial}

In this section we complete the proof of the two leg correspondence, modulo a combinatorial identity (Theorem \ref{thm:comb}, proved in Section \ref{sec:comb}), by proving that the correspondence holds for the inital values $w_s$ and $w_a$.

\subsubsection{Symmetric Case}

Recall that $w_s=(\frac{1}{n},-\frac{1}{n},0)$.  Since $w_3=0$, all terms $V_{g,\gamma}(\tau^+,\tau^-,\emptyset;w_a)$ vanish by definition unless either
\begin{itemize}
\item $l(\tau^{\pm})=1$ and $l(\tau^{\mp})=0$, or
\item $l(\tau^+)=l(\tau^-)=1$, $|\tau^+|=|\tau^-|=d$, $g=0$, and $\gamma=\emptyset$.
\end{itemize}

In the first case, by utilizing the $\bZ_n$-Mumford relation, the fact that $\Mbar_{g,\gamma}(\cB\bZ_n)\rightarrow\Mbar_{g,|\gamma|}$ has degree $n^{2g-1}$, and the string equation, we compute:

\[
\tilde V_{g,\gamma}^\alpha((d),\emptyset,\emptyset;w_s)=
\begin{cases}
(\alpha_+)^{d+1}\frac{\sqrt{-1}(-1)^{|\gamma|+\frac{d}{n}-\sum_i\frac{i}{n}m_i(\gamma)}d^{2g-2+|\gamma|}}{n^{|\gamma|-1}}\int_{\Mbar_{g,1}}\lambda_g\psi^{2g-2} & d=\sum m_i(\gamma)i \text{ mod } n\\
0 & \text{else.}
\end{cases}
\]
and
\[
\tilde V_{g,\gamma}^\alpha(\emptyset,(d),\emptyset;w_s)=\begin{cases}
(\alpha_-)^{d+1}\frac{\sqrt{-1}(-1)^{\frac{d}{n}+\sum_i\frac{i}{n}m_i(\gamma)}d^{2g-2+|\gamma|}}{n^{|\gamma|-1}}\int_{\Mbar_{g,1}}\lambda_g\psi^{2g-2} & -d=\sum m_i(\gamma)i \text{ mod } n\\
0 & \text{else}
\end{cases}
\]

In the second case, we compute straight from the definitions:
\[
\tilde V^{\alpha}_{0,\emptyset}((d),(d),\emptyset;w_s):=\lim_{w\rightarrow w_s} \tilde V^{\alpha}_{0,\emptyset}((d),(d),\emptyset;w)=\frac{\alpha_+\alpha_-}{d}.
\]

We now use these computations to prove the symmetric correspondence at framing $w_s$.  The following computations are similar to those in \cite{z:gmvf}.  We start by writing
\begin{equation}\label{ws:1}
\tilde V^{\alpha}_{(d),\emptyset,\emptyset}(x,u;w_s)=\sum_{g,\gamma} \sum_{j=0}^{n-1}\frac{\xi_n^{j\left(d-\sum_im_i(\gamma)i\right)}}{n}\tilde V^{\alpha}_{g,\gamma}((d),\emptyset,\emptyset;w_s)u^{2g-1}\frac{x^\gamma}{\gamma!}
\end{equation}
where the sum over $j$ simply encodes the necessary condition on the degree and twistings.  Applying the Faber-Pandharipande identity \cite{fp:hiagwt}:
\[
\sum_gt^{2g}\int_{\Mbar_{g,1}}\lambda_g\psi^{2g-2}=\frac{t}{2}\csc\left(\frac{t}{2}\right),
\]
\eqref{ws:1} becomes
\begin{equation}\label{ws:2}
(\alpha_+)^{d+1}\frac{\sqrt{-1}\xi_{2n}^d}{2d}\sum_{j=0}^{n-1}\xi_n^{jd}(E_j(x))^d\csc\left(\frac{du}{2}\right)
\end{equation}
where
\[
E_j(x):=\exp\left(-\sum_{i=1}^{n-1}\frac{\xi_{2n}^{-i(2j+1)}}{n}x_i\right)
\]
Introducing variables $y^+=(y_0^+,y_1^+,\dots)$ such that $p_d^+=\sum (-y_i^+)^d$, \eqref{ws:2} implies that
\begin{align}\label{ws:3}
\nonumber\sum_d \tilde V^\alpha_{(d),\emptyset,\emptyset}(x,u;w_s) p_d^+&=\alpha_+\sum_{d\geq 0}\sum_{j=0}^{n-1}\sum_{k,l\geq 0}\frac{1}{d}\left(-\alpha_+\xi_{2n}\xi_{n}^jE_j(x)e^{\sqrt{-1}\frac{u}{2}}e^{\sqrt{-1}uk}y_l^+\right)^d \\
&=-\alpha_+\log\left( \prod_{j=0}^{n-1}\prod_{k,l\geq 0}\left(1+\alpha_+\xi_{2n}\xi_{n}^jE_j(x)e^{\sqrt{-1}\frac{u}{2}}e^{\sqrt{-1}uk}y_l^+ \right)\right).
\end{align}
After the prescribed identification of variables, we have \cite{z:gmvf}:
\[
\xi_{n}^jE_j(x)=-\xi_{2n}^{-1}q_1^{-\frac{1}{n}}\cdots q_{n-1}^{-\frac{n-1}{n}}q_{j+1}\cdots q_{n-1},
\]
so \eqref{ws:3} becomes
\begin{equation}\label{ws:4}
-\alpha_+\log\left(\prod_{j=0}^{n-1}\prod_{k,l\geq 0}\left(1-\alpha_+q^{\frac{1}{2}}q_1^{-\frac{1}{n}}\cdots q_{n-1}^{-\frac{n-1}{n}}q_{j+1}\cdots q_{n-1}q^ky_l^+ \right)\right)
\end{equation}
Finally, applying the Cauchy identity to \eqref{ws:4} we have:
\begin{align}\label{ws:5}
\nonumber\sum_d \tilde V^\alpha_{(d),\emptyset,\emptyset}(x,u;w_s) p_d^+&=\log\left( \sum_{d\geq1}\left(q^{\frac{1}{2}}q_1^{-\frac{1}{n}}\cdots q_{n-1}^{-\frac{n-1}{n}} \right)^d\sum_{|\omega|=d} \overline{s_{\omega}(\fq_\bullet)}s_{\omega^{\alpha_+}}(y^+) \right)\\
&=\log\left( \sum_{\omega} \overline{s_{\omega}(\tilde \fq_\bullet)}s_{\omega^{\alpha_+}}(y^+) \right)
\end{align}
where $\tilde \fq_\bullet:=q^{\frac{1}{2}}q_1^{-\frac{n-1}{n}}\cdots q_{n-1}^{-\frac{1}{n}} \fq_\bullet.$

Similarly, if we introduce variables $y^-=(y_0^-,y_1^-,\dots)$ such that $p_d^-=\sum (-y_i^-)^d$,
\begin{equation}\label{ws:6}
\sum_d \tilde V^\alpha_{\emptyset,\emptyset,(d)}(x,u;w_s) p_d^-=\log\left( \sum_\omega s_{\omega}(\tilde \fq_\bullet)s_{\omega^{\alpha_-}}(y^-) \right)
\end{equation}
and
\begin{equation}\label{ws:7}
\sum_{d} \tilde V^\alpha_{\emptyset,(d),(d)}(x,u;w_s) p_d^+p_d^-=\log\left( \sum_\omega s_{\omega^{\alpha^+}}(y^+)s_{\omega^{\alpha^-}}(y^-) \right).
\end{equation}
Pulling together equations \eqref{ws:5}, \eqref{ws:6}, and \eqref{ws:7} and exponentiating to pass to the disconnected series, we compute that $\tilde V^{\bullet,\alpha}_{\emptyset,\tau^+,\tau^-}(x,u;w_s)$ (after the identification of variables) is equal to the coefficient of $p_{\tau^+}p_{\tau^-}$ in the following expression:
\begin{align*}
&\left(\sum_{\omega^+} \overline{s_{\omega^+}(\tilde \fq_\bullet)}s_{(\omega^+)^{\alpha_+}}(y^+)\right)\left( \sum_\omega s_{\omega^{\alpha_+}}(y^+)s_{\omega^{\alpha_-}}(y^-) \right)\left( \sum_{\omega^-} s_{\omega^-}(\tilde \fq_\bullet)s_{(\omega^-)^{\alpha_-}}(y^-) \right)\\
&=\sum_{\rho^+,\rho^-,\omega^+,\omega^-,\omega}\overline{s_{\omega^+}(\tilde \fq_\bullet)}c_{(\omega^+)^{\alpha_+},\omega^{\alpha_+}}^{\rho^+} s_{\rho^+}(y^+)c_{(\omega^-)^{\alpha_-},\omega^{\alpha_-}}^{\rho^-} s_{\rho^-}(y^-)s_{\omega^-}(\tilde \fq_\bullet)\\
&=\sum_{\rho^+,\rho^-,\omega}\overline{s_{(\rho^+)^{\alpha_+}/\omega}(\tilde \fq_\bullet)} s_{(\rho^-)^{\alpha_-}/\omega}(\tilde \fq_\bullet)s_{\rho^+}(y^+)s_{\rho^-}(y^-)
\end{align*}
where $c_{\cdot,\cdot}^\cdot$ denotes Littlewood-Richardson coefficients.  To extract the coefficient of $p_{\tau^+}p_{\tau^-}$, we simply write the Schur functions in terms of power sum functions and we obtain
\begin{align*}
\tilde V^{\bullet,\alpha}_{\emptyset,\tau^+,\tau^-}(w_s)&=\sum_{\rho^+,\rho^-}\left(\sum_\omega \overline{s_{(\rho^+)^{\alpha_+}/\omega}(\tilde \fq_\bullet)} s_{(\rho^-)^{\alpha_-}/\omega}(\tilde \fq_\bullet) \right)(-1)^{|\rho^+|}\frac{\chi_{\rho^+}(\tau^+)}{z_{\tau^+}}(-1)^{|\rho^-|}\frac{\chi_{\rho^-}(\tau^-)}{z_{\tau^-}}\\
&=\sum_{\rho^+,\rho^-}\tilde P^\alpha_{\rho^+,\rho^-,\emptyset}(w_s)\frac{\chi_{\rho^+}(\tau^+)}{z_{\tau^+}}\frac{\chi_{\rho^-}(\tau^-)}{z_{\tau^-}}
\end{align*}
which completes the proof of the symmetric correspondence.

\subsubsection{Asymmetric Case}

Recall that $w_a=(\frac{1}{n},-1-\frac{1}{n},1)$.  The particular choice of initial condition $w_a$ is motivated by the observation of the following identity which follows immediately from the definitions:
\begin{equation}\label{conjsym}
\tilde V^\alpha_{\tau+d,\emptyset,\mu}(w_a)=n\alpha_+^{d+1}(-1)^{d+1}\xi_{2n}^{-d}\frac{|\Aut(\tau)||\Aut(\mu+d)|}{|\Aut(\tau+d)||\Aut(\mu)|}\tilde V^\alpha_{\tau,\emptyset,\mu+d}(w_a)
\end{equation}
where $\tau+d$ is the partition obtained from $\tau$ by adding a part of size $d$ and $\mu+d$ is the $n$-partition obtained from $\mu$ by adding a part of size $d$ to $\mu^k$ where $k=d\text{ mod } n$.  In other words, \eqref{conjsym} can be used to move the parts from $\tau$ over to $\mu$, one by one, until $\tau=\emptyset$.  When $\tau=\emptyset$, the correspondence holds by the main theorem in \cite{rz:ggmv}.  Therefore, to prove the correspondence for $\tau\neq\emptyset$ we need only show that $\tilde P^\alpha_{\emptyset,\rho,\lambda}(w_a)$ satisfies the representation basis analog of \eqref{conjsym}.  In other words, we need to prove that
\begin{align*}
\sum_{\rho,\lambda}&\tilde P^\alpha_{\rho,\emptyset,\lambda}(w_a)\frac{\chi_\rho(\tau+d)}{z_{\tau+d}}\frac{\chi_\lambda(\mu)}{z_\mu}\\
&=n\alpha_+^{d+1}(-1)^{d+1}\xi_{2n}^{-d}\frac{|\Aut(\tau)||\Aut(\mu+d)|}{|\Aut(\tau+d)||\Aut(\mu)|}\sum_{\rho,\lambda}\tilde P^\alpha_{\rho,\emptyset,\lambda}(w_a)\frac{\chi_\rho(\tau)}{z_{\tau}}\frac{\chi_\lambda(\mu+d)}{z_{\mu+d}}\\
&=\alpha_+^{d+1}(-1)^{d+1}\xi_{2n}^{-d}\sum_{\lambda,\rho}\tilde P^\alpha_{\rho,\emptyset,\lambda}(w_a)\frac{\chi_\rho(\tau)}{z_{\tau+d}}\frac{\chi_\lambda(\mu+d)}{z_\mu}
\end{align*}
where the second equality follows from formulae for the order of centralizers given by equations \eqref{orderofcent1} and \eqref{orderofcent2}. We should notice that if we can prove the above identity for $\alpha_+=1$, then the identity for $\alpha_+=-1$ is also true by the fact that $\chi_{\rho'}(\tau)=(-1)^\tau\chi_{\rho}(\tau)$. Now let us prove this identity for $\alpha_+=1$.  By definition of $\tilde P^\alpha_{\rho,\emptyset,\lambda}(w_a)$, this identity becomes
\begin{align}\label{reduction}
\nonumber\sum_{\rho,\lambda}&\chi_\rho(\tau+d)\left( \prod_k \xi_n^{-k|\lambda_k|}\frac{\chi_{\bar\lambda}(n^{|\lambda|})}{\dim(\lambda)}\chi_\lambda(\mu) \right)  \hat s_{\lambda}(\bq)\overline{\hat s_{\rho}(\fq_{\bullet-\lambda})}\\
&=\left(q_1^\frac{n-1}{n}\cdots q_{n-1}^\frac{1}{n} \right)^d \sum_{\rho,\lambda}\chi_\rho(\tau)\left( \prod_k \xi_n^{k|\lambda_k|}\frac{\chi_{\bar\lambda}(n^{|\lambda|})}{\dim(\lambda)}\chi_\lambda(\mu+d) \right)  \hat s_{\lambda}(\bq)\overline{\hat s_{\rho}(\fq_{\bullet-\lambda})}
\end{align}
where we define
\[
\hat s_\lambda(\bq):=\left(\prod_{(i,j)\in\bar\lambda}q_{j-i}^\frac{i-j}{n}\right) s_\lambda(\bq)
\]
and
\[
\hat s_\rho(\fq_{\bullet-\lambda}):=\left(\prod_{(i,j)\in\rho}q^{i-j}\right) s_\rho(\fq_{\bullet-\lambda}).
\]
The identity \eqref{reduction} is implied by the next theorem.
\begin{theorem}\label{thm:comb}
For any partition $\omega$ and $n$-partition $\sigma$,
\begin{align*}
\hat s_{\sigma}(\bq)\sum_{|\rho|=|\omega|+d}&\overline{\hat s_\rho(\fq_{\bullet-\sigma})}\chi_{\rho\setminus\omega}(d)\\
&=(q_1^{\frac{1}{n}}\cdots q_{n-1}^{\frac{n-1}{n}})^d\sum_{|\lambda|=|\sigma|+n}\hat s_{\lambda}(\bq)\overline{\hat s_\omega(\fq_{\bullet-\lambda})}\chi_{\bar\lambda\setminus\bar\sigma}(nd)
\end{align*}
where the relative character $\chi_{\rho\setminus\omega}(d)$ for two Young diagrams $\rho$ and $\omega$ is equal to
\[
\chi_{\rho\setminus\omega}(d)=
\begin{cases}
(-1)^{k-1} &\text{ if $\rho\setminus\omega$ is a connected border strip of height $k$,}\\
0 &\text{ else.}
\end{cases}
\]
\end{theorem}

We prove Theorem \ref{thm:comb} in the next section.  In the meantime, note that Theorem \ref{thm:comb} implies \eqref{reduction} by summing over $\omega$ and $\sigma$ and using the identities
\[
\chi_\rho(\tau+d)=\sum_\omega\chi_{\rho\setminus\omega}(d)\chi_\omega(\tau)
\]
and
\[
\prod_k \xi_n^{k|\lambda_k|}\frac{\chi_{\bar\lambda}(n^{|\lambda|})}{\dim(\lambda)}\chi_\lambda(\mu+d)=\sum_{\sigma}\chi_{\bar\lambda\setminus\bar\sigma}(nd)\prod_k \xi_n^{k|\sigma_k|}\frac{\chi_{\bar\sigma}(n^{|\sigma|})}{\dim(\sigma)}\chi_\sigma(\mu).
\]
This completes the proof of the asymmetric correspondence.

\section{Proof of Theorem \ref{thm:comb}}\label{sec:comb}

\subsection{Overview}
Notice that the sums in Theorem \ref{thm:comb} correspond to adding strips to Young diagrams.  On the left, we are summing over all ways of adding $d$ strips to $\omega$ and on the right we are summing over all ways of adding $nd$ strips to $\sigma$.  To prove the identity, we begin by expressing all of the Schur functions in terms of determinants of certain matrices.  Each summand in the left side of Theorem \ref{thm:comb} can be obtained by modifying a column of the matrix and each summand in the right side is obtained by modifying a row.  Expanding the determinants along the corresponding columns and rows proves the theorem.

\subsection{Determinantal Expressions}

\subsubsection{Schur Functions}

By the classical definition of the Schur functions \cite{m:sfhp}, we have
\begin{align}\label{schurdet}
\nonumber \overline{\hat s_\omega(\fq_{\bullet-\sigma})}&=\prod_{(i,j)\in\omega}q^{i-j}\lim_{m\rightarrow\infty}\frac{\det\left( \left( q_0^{-1}\cdots q_{\bar\sigma_i-i}^{-1} \right)^{m-j+\omega_j} \right)_{1\leq i,j \leq m}}{\det \left( \left( q_0^{-1}\cdots q_{\bar\sigma_i-i}^{-1} \right)^{m-j} \right)_{1\leq i,j\leq m}}\\
&=\prod_{(i,j)\in\omega}q^{i-j}\lim_{m\rightarrow\infty}q^{\frac{m}{n}|\omega|}\frac{\det\left( \left( q_0^{-1}\cdots q_{\bar\sigma_i+m-i}^{-1} \right)^{m-j+\omega_j} \right)_{1\leq i,j \leq m}}{\det \left( \left( q_0^{-1}\cdots q_{\bar\sigma_i+m-i}^{-1} \right)^{m-j} \right)_{1\leq i,j\leq m}}
\end{align}
where we adopt the convention
\[
q_0^{-1}\cdots q_k^{-1}=\begin{cases}
1 & k=-1\\
q_{n-1}\cdots q_{n+k-1} & k<-1.
\end{cases}
\]

\subsubsection{Loop Schur Functions}

For $l>0$, define functions
\[
\hat h_l^r:=\frac{q_r^{-\frac{r}{n}}\cdots q_{r+l-1}^{-\frac{r+l-1}{n}}}{\prod_{i=1}^l\left( 1-q_{l-i+r}\cdots q_{l-1+r} \right)}.
\]
We define $\hat h_0^r=1$ and $\hat h_l^r=0$ for $l<0$.

\begin{lemma}\label{loopdet}
For $m\geq l(\bar\sigma)$,
\[
\hat s_\sigma(\bq)=\det\left( \hat h_{\bar\sigma_i-i+j}^{1-j} \right)_{1\leq i,j \leq m}
\]
as rational functions in $\bq$.
\end{lemma}

\begin{proof}
We first remark that the Jacobi-Trudi identity generalizes to loop Schur functions \cite{l:lsfafmp}.  In particular, if we define the complete homogeneous loop functions in the variables $x=\{x_{i,j}:i\in\bZ_n,j\geq 1\}$ as
\[
h_l^r:=\sum_{1\leq i_1\leq \cdots \leq i_l} x_{r,i_1}\cdots x_{r+l-1,i_l},
\]
then for any $m>l(\bar\sigma)$,
\begin{equation}\label{jactrud}
s_{\sigma}(x)=\det\left( h_{\bar\sigma_i-i+j}^{1-j} \right)_{1\leq i,j \leq m}.
\end{equation}
The proof of \eqref{jactrud} is a straightforward generalization of the lattice path argument which proves the classical Jacobi-Trudi identity \cite{s:ec2}.  Moreover, since $h_0^r=1$ and $h_k^r=0$ for $k<0$, then
\begin{equation}\label{jactrud2}
\det\left( h_{\bar\sigma_i-i+j}^{1-j} \right)_{1\leq i,j \leq m}=\det\left( h_{\bar\sigma_i-i+j}^{1-j} \right)_{1\leq i,j \leq l(\bar\sigma)}.
\end{equation}
Notice that after the substitution, $x_{i,j}=q_i^j$, we have
\[
h_{\bar\sigma_i-i+j}^{1-j}=q_{1-j}^{\frac{1-j}{n}}\cdots q_{\bar\sigma_i-i}^{\frac{\bar\sigma_i-i}{n}}\hat h_{\bar\sigma_i-i+j}^{1-j}.
\]
Therefore,
\begin{align*}
s_\sigma(\bq)&=\det\left( q_{1-j}^{\frac{1-j}{n}}\cdots q_{\bar\sigma_i-i}^{\frac{\bar\sigma_i-i}{n}}\hat h_{\bar\sigma_i-i+j}^{1-j} \right)_{1\leq i,j \leq l(\bar\sigma)}\\
&=\prod_{i=1}^{l(\bar\sigma)} q_1^\frac{1}{n}\cdots q_{\bar\sigma_i-i}^{\frac{\bar\sigma_i-i}{n}}\prod_{j=1}^ {l(\bar\sigma)} q_1^{-\frac{1}{n}}\cdots q_{1-j}^{\frac{1-j}{n}}\det\left( \hat h_{\bar\sigma_i-i+j}^{1-j} \right)_{1\leq i,j \leq l(\bar\sigma)}\\
&=\prod_{(i,j)\in\bar\sigma}q_{j-i}^{\frac{j-i}{n}}\det\left( \hat h_{\bar\sigma_i-i+j}^{1-j} \right)_{1\leq i,j \leq m}.
\end{align*}

\end{proof}

From this point on, we take $m\geq l(\bar\sigma)$ to be a multiple of $n$.

Define
\[
f_{a,b}^r:=q_r^{\frac{r}{n}}\cdots q_{r+b-1}^{\frac{r+b-1}{n}}\prod_{i=1}^b\left( 1-q_{r+i-1}\cdots q_{r+a-1} \right).
\]
For $l\leq m$ we have the relation
\[
\hat h_l^r = \hat h_m^{r+l-m}f_{m,m-l}^{r+l-m}.
\]
In particular,
\[
\hat h_{\bar\sigma_i-i+j}^{1-j} = \hat h_{\bar\sigma_i-i+m}^{1-m}f_{\bar\sigma_i-i+m,m-j}^{1-m}.
\]
Inserting this into Lemma \ref{loopdet} gives
\begin{equation}\label{loopdet2}
\hat s_\sigma(\bq)=\prod_{i=1}^m \hat h_{\bar\sigma_i-i+m}^{1-m}\det\left( f_{\bar\sigma_i-i+m,m-j}^{1-m}\right)_{1\leq i,j \leq m}.
\end{equation}
We can rewrite
\[
f_{a,b}^r=q_r^{\frac{r}{n}}\cdots q_{r+b-1}^{\frac{r+b-1}{n}}\sum_{l=0}^b(-1)^le_l\left( q_r\cdots q_{r+a-1},\dots,q_{r+b-1}\cdots q_{a+r-1}\right)
\]
where the $e_l$ are elementary symmetric functions.  Therefore,
\begin{align}\label{loopdet3}
\nonumber f_{\bar\sigma_i-i+m,m-j}^{1-m}&=q_1^{\frac{1-m}{n}}\cdots q_{m-j}^{-\frac{j}{n}}\sum_{l=0}^{m-j}(-1)^le_l\left( q_1\cdots q_{\bar\sigma_i+m-i},\dots, q_{m-j}\cdots q_{\bar\sigma_i+m-i} \right)\\
&=q_1^{\frac{1-m}{n}}\cdots q_{m-j}^{-\frac{j}{n}}\sum_{l=0}^{m-j}\left(-q_1\cdots q_{\bar\sigma_i+m-i}  \right)^l e_l\left(1,q_1^{-1},q_1^{-1}q_2^{-1},...,q_1^{-1}\cdots q_{m-j-1}^{-1}  \right)
\end{align}
where the first equality uses the assumption that $n$ divides $m$.  Substituting \eqref{loopdet3} into \eqref{loopdet2} and simplifying using elementary column operations, we obtain
\begin{align}\label{loopdet4}
\nonumber \hat s_\sigma(\bq)=&\prod_{i=1}^m \hat h_{\bar\sigma_i-i+m}^{1-m}\prod_{j=1}^m q_1^{\frac{1-m}{n}}\cdots q_{m-j}^{-\frac{j}{n}}\\
&\cdot \prod_{j=1}^{m-1}(-1)^je_l\left( 1,q_1^{-1},q_1^{-1}q_2^{-1},...,q_1^{-1}\cdots q_{m-j-1}^{-1} \right)\\
&\nonumber \cdot \det \left( \left( q_1\cdots q_{\bar\sigma_i+m-i} \right)^{m-j} \right)_{1\leq i,j\leq m}.
\end{align}

By the hook-length formula for loop Schur functions, we have an equality as rational functions:
\begin{equation}\label{loopdet5}
\hat s_\sigma(\bq)=(-1)^{|\bar\sigma|}q^{-\frac{|\bar\sigma|}{n}}\prod_{(i,j)\in\bar\sigma}q_{j-i}^{2\frac{i-j}{n}+i-j}\hat s_\sigma\left(\bq^{-1}\right).
\end{equation}
Substituting \eqref{loopdet4} into the right side of \eqref{loopdet5}, we obtain the following equality of rational functions:
\begin{align}\label{loopdet6}
\nonumber \hat s_\sigma(\bq)=&q_0^{\frac{m(m-1)}{2}}(-1)^{|\bar\sigma|}q^{-\frac{|\bar\sigma|}{n}}\prod_{(i,j)\in\bar\sigma}q_{j-i}^{2\frac{i-j}{n}+i-j}\prod_{i=1}^m \hat h_{\bar\sigma_i-i+m}^{1-m}\left(\bq^{-1}\right)\\
&\cdot \prod_{j=1}^m q_1^{\frac{1-m}{n}}\cdots q_{m-j}^{-\frac{j}{n}}\prod_{j=1}^{m-1}(-1)^je_l\left( 1,q_1,q_1q_2,...,q_1\cdots q_{m-j-1} \right)\\
\nonumber&\cdot \det \left( \left(q_0^{-1} q_1^{-1}\cdots q_{\bar\sigma_i+m-i}^{-1} \right)^{m-j} \right)_{1\leq i,j\leq m}.
\end{align}

\subsection{Proof of Theorem \ref{thm:comb}}

Theorem \ref{thm:comb} is equivalent to the following identity
\begin{equation}\label{combident1}
(q_1^{\frac{1}{n}}\cdots q_{n-1}^{\frac{n-1}{n}})^{-d}\sum_{|\rho|=|\omega|+d}\overline{\hat s_\rho(\fq_{\bullet-\sigma})}\chi_{\rho\setminus\omega}(d)
=\sum_{|\lambda|=|\sigma|+d}\frac{\hat s_{\lambda}(\bq)}{\hat s_{\sigma}(\bq)}\overline{\hat s_\omega(\fq_{\bullet-\lambda})}\chi_{\bar\lambda\setminus\bar\sigma}(nd)
\end{equation}

We now compute both sides of \eqref{combident1} using the determinantal expressions from the previous section.

\subsubsection{Left side of \eqref{combident1}}

By the determinantal expression for Schur functions \eqref{schurdet}, the left side of \eqref{combident1} is obtained in the limit $m\rightarrow \infty$ from
\begin{align}\label{combident2}
&\nonumber \left(q_1^{\frac{1}{n}}\cdots q_{n-1}^{\frac{n-1}{n}}\right)^{-d}\sum_{|\rho|=|\omega|+d}   \left(\prod_{(i,j)\in\rho}q^{i-j}q^{\frac{m}{n}|\rho|}\frac{\det\left( \left( q_0^{-1}\cdots q_{\bar\sigma_i+m-i}^{-1} \right)^{m-j+\rho_j} \right)_{1\leq i,j \leq m}}{\det \left( \left( q_0^{-1}\cdots q_{\bar\sigma_i+m-i}^{-1} \right)^{m-j} \right)_{1\leq i,j\leq m}}\right)    \chi_{\rho\setminus\omega}(d)\\
&=\sum_{t=1}^m\left(q_1^{\frac{1}{n}}\cdots q_{n-1}^{\frac{n-1}{n}}\right)^{-d}q^{\frac{m}{n}(|\omega|+d)}\prod_{(i,j)\in\omega}q^{i-j}\prod_{l=1}^d q^{-(\omega_t-t+l)} \frac{\det\left( \left( q_0^{-1}\cdots q_{\bar\sigma_i+m-i}^{-1} \right)^{m-j+\omega_j+d\delta_{j,t}} \right)_{1\leq i,j \leq m}}{\det \left( \left( q_0^{-1}\cdots q_{\bar\sigma_i+m-i}^{-1} \right)^{m-j} \right)_{1\leq i,j\leq m}}.
\end{align}
The second line follows from the observation that adding a $d$-strip to $\omega$ corresponds to adding $d$ to a single part of $\omega$.  The sign $\chi_{\rho\setminus\omega}(d)$ corresponds to reordering the columns in the numerator until the exponents are strictly decreasing.

Expanding the determinant in the numerator in the right side of \eqref{combident2} along the $t$th column, we obtain
\begin{align}\label{combident3}
\sum_{s,t=1}^m(-1)^{s+t}\nonumber\left(q_1^{\frac{1}{n}}\cdots q_{n-1}^{\frac{n-1}{n}}\right)^{-d}&q^{\frac{m}{n}(|\omega|+d)}\prod_{(i,j)\in\omega}q^{i-j}\prod_{l=1}^d q^{-(\omega_t-t+l)}\left( q_0^{-1}\cdots q_{\bar\sigma_s+m-s}^{-1} \right)^{m-t+\omega_t+d}\\
&\cdot \frac{\det\left( \left( q_0^{-1}\cdots q_{\bar\sigma_i+m-i}^{-1} \right)^{m-j+\omega_j} \right)_{i\neq s, j\neq t}}{\det \left( \left( q_0^{-1}\cdots q_{\bar\sigma_i+m-i}^{-1} \right)^{m-j} \right)_{1\leq i,j\leq m}}.
\end{align}

\subsubsection{Right side of \eqref{combident1}}

The sum in the right side of \eqref{combident1} is over all ways of adding $nd$ strips to $\bar\sigma$.  Analogous to the discussion for the left side, each summand corresponds to adding $nd$ to a single part of $\bar\sigma$.  By this observation and the determinantal expressions for Schur functions \eqref{schurdet} and loop Schur functions \eqref{loopdet5}, the right side of \eqref{combident1} is obtained in the limit $m\rightarrow\infty$ from
\begin{align}\label{combident4}
\nonumber\sum_{s=1}^m & (-1)^{nd}q^{-d}\prod_{l=1}^{nd}q_{\bar\sigma_s-s+l}^{-(\bar\sigma_s-s+l)(\frac{2}{n}+1)}q^{\frac{m}{n}|\omega|}\prod_{(i,j)\in\omega} q^{i-j}\frac{\hat h_{\bar\sigma_s-s+m+nd}^{1-m}(\bq^{-1})}{\hat h_{\bar\sigma_s-s+m}^{1-m}(\bq^{-1})}\\
&\cdot\frac{\det \left( \left(q_0^{-1} \cdots q_{\bar\sigma_i+m-i+nd\delta_{i,s}}^{-1} \right)^{m+\omega_j-j} \right)_{1\leq i,j\leq m}}{\det \left( \left(q_0^{-1} \cdots q_{\bar\sigma_i+m-i+nd\delta_{i,s}}^{-1} \right)^{m-j} \right)_{1\leq i,j\leq m}}
\end{align}
where we can simplify
\begin{equation}\label{combident6}
\frac{\hat h_{\bar\sigma_s-s+m+nd}^{1-m}(\bq^{-1})}{\hat h_{\bar\sigma_s-s+m}^{1-m}(\bq^{-1})}=\frac{\prod_{l=1}^{nd}q_{\bar\sigma_s-s+l}^{\frac{1}{n}(\bar\sigma_s-s+l)}}{\prod_{l=1}^{nd}\left( 1-q_l^{-1}\cdots q_{\bar\sigma_s-s+m+nd}^{-1} \right)}.
\end{equation}

Expanding the determinant in the numerator of \eqref{combident4} along the $s$th row and incorporating \eqref{combident6}, we obtain
\begin{align}\label{combident5}
\nonumber\sum_{s,t=1}^m(-1)^{s+t+nd}\prod_{l=1}^{nd}&q_{\bar\sigma_s-s+l}^{-(\bar\sigma_s-s+l)(\frac{1}{n}+1)}q^{\frac{m}{n}|\omega|-d}\prod_{(i,j)\in\omega} q^{i-j}\frac{\left( q_0^{-1}\cdots q_{\bar\sigma_s+m-s+nd}^{-1} \right)^{m+\omega_t-t}}{\prod_{l=1}^{nd}\left( 1-q_l^{-1}\cdots q_{\bar\sigma_s-s+m+nd}^{-1} \right)}\\
&\cdot\frac{\det\left( \left( q_0^{-1}\cdots q_{\bar\sigma_i+m-i}^{-1} \right)^{m-j+\omega_j} \right)_{i\neq s, j\neq t}}{\det \left( \left( q_0^{-1}\cdots q_{\bar\sigma_i+m-i}^{-1} \right)^{m-j} \right)_{1\leq i,j\leq m}}.
\end{align}

It is left to prove that \eqref{combident3} and \eqref{combident5} are equal in the $m\rightarrow\infty$ limit as series centered at $\bq=0$.  In other words, for any number $N$, we must show that there is a corresponding $M$ so that the series \eqref{combident3} and \eqref{combident5} are equal modulo terms of degree greater than $N$.

We begin by proving that for $s$ large relative to $m$, the summands in \eqref{combident3} and \eqref{combident5} do not contribute to the large $m$ limit.  To this end, define
\[
D_{s,t}:=\frac{\det\left( \left( q_0^{-1}\cdots q_{\bar\sigma_i+m-i}^{-1} \right)^{m-j+\omega_j} \right)_{i\neq s, j\neq t}}{\det \left( \left( q_0^{-1}\cdots q_{\bar\sigma_i+m-i}^{-1} \right)^{m-j} \right)_{1\leq i,j\leq m}}.
\]
and let $d_{s,t}$ be the least degree of the terms in the $\bq=0$ expansion of $D_{s,t}$.

\begin{lemma}
\[
d_{s,t}> -\sum_{l=1}^{l(\omega)}\omega_l(\bar\sigma_l+m-l+1)+(\bar\sigma_s+m-s+1)(m-t+\omega_t).
\]
\end{lemma}

\begin{proof}
Since all of the matrix entries have negative degree and the degrees strictly grow along rows and columns, the least degree term in the expansion of the determinant is the product of the main diagonal entries.  From this we see that the denominator in $D_{s,t}$ has degree $-\sum_{l=1}^m(\bar\sigma_l+m-l+1)(m-l)$.  Moreover, the denominator is a Vandermonde determinant and we see that by multiplying through by $\prod_{l=1}^m \left(q_0\cdots q_{\bar\sigma_l+m-l}\right)^{m-l}$, we can expand it as a product of geometric series centered at $\bq=0$.  Therefore, $d_{s,t}$ is the least degree term in the expansion of the numerator, plus $\sum_{l=1}^m(\bar\sigma_l+m-l+1)(m-l)$.

We now compute the degree of the determinant in the numerator.  We assume $t<s$, the other cases are similar.  Multiplying along the main diagonal, we find the least degree term in the numerator to be
\[
-\sum_{l\leq t-1 \atop l\geq s+1}(\bar\sigma_l+m-l+1)(m-l+\omega_l)-\sum_{l=t}^{s-1}(\bar\sigma_l+m-l+1)(m-(l+1)+\omega_{l+1}).
\]
Adding $\sum_{l=1}^m(\bar\sigma_l+m-l+1)(m-l)$ to this, we get
\[
d_{s,t}=-\sum_{l=1}^{l(\omega)}\omega_l(\bar\sigma_l+m-l+1)+\sum_{l=t}^{s-1}(\bar\sigma_l+m-l+1)(\omega_l-\omega_{l+1}+1)+(\bar\sigma_s+m-s+1)(m-s+\omega_s).
\]
We have
\begin{align*}
\sum_{l=t}^{s-1}(\bar\sigma_l+m-l+1)(\omega_l-\omega_{l+1}+1)&> \sum_{l=t}^{s-1}(\bar\sigma_s+m-s+1)(\omega_l-\omega_{l+1}+1)\\
&=(\bar\sigma_s+m-s+1)(\omega_t-\omega_s+s-t)
\end{align*}
which finishes the proof.
\end{proof}

Now let $D_{s,t}^l$ and $D_{s,t}^r$ denote the $s,t$ summand of \eqref{combident3} and \eqref{combident5}, respectively, and let $d_{s,t}^l$ and $d_{s,t}^r$ denote the least degree of the terms in the $\bq=0$ expansions. Similarly let $C_{s,t}^l$ and $C_{s,t}^r$ be the coefficents of \eqref{combident3} and \eqref{combident5} respectively so that $D_{s,t}^l=C_{s,t}^lD_{s,t}$ and $D_{s,t}^r=C_{s,t}^rD_{s,t}$, and let $c_{s,t}^l$ and $c_{s,t}^r$ be the least degree of the terms in the $\bq=0$ expansions.

\begin{corollary}\label{cor:bound}
\[
\lim_{s\rightarrow\infty} d_{s,t}^l=\lim_{s\rightarrow\infty} d_{s,t}^r=\infty
\]
More specifically, for every number $N$, there is a corresponding $S$ such that for every $s,m\geq S$ and any $t$, $d_{s,t}^l,d_{s,t}^r> N$.
\end{corollary}

\begin{proof}
Add the lower bound for $d_{s,t}$ to $c_{s,t}^l$ and $c_{s,t}^r$.  In both cases, the result is bounded below by a function of $s$ which grows without bound.  In the case of \eqref{combident5}, the denominator of the prefactor must be centered at $0$ and then expanded as a geometric series.  In the expansion, only the initial term needs to be considered because all others will have higher degree.
\end{proof}

Now let $N$ be any number.  Choose $S$ satisfying the conclusion of Corollary \ref{cor:bound}.  Choose M so that $(\bar\sigma_s+M-S+nd)\cdot d_{S,t}^r> N$ for all $t$.  The terms $D_{s,t}^l$ and $D_{s,t}^r$ with $s\geq S$ only contribute terms of degree greater than $N$.  Moreover, for $s<S$, all terms in the $\bq=0$ expansion of $C_{s,t}^r$ have degree greater than $N$, except for the first.  It is then a tedious exercise to show that the initial term in the $\bq=0$ expansion of $C_{s,t}^r$ is equal to $C_{s,t}^l$.  Therefore, for our choice of $M$, \eqref{combident3} and \eqref{combident5} are equal modulo terms of degree greater than $N$.

\section{Gluing}\label{sec:gluing}

In this section we show that Conjecture \ref{conjecture} is compatible with the orbifold vertex gluing algorithms in GW and DT theory.  First, we set notation and recall the gluing algorithms.

\subsection{Notation}

Let $\cX$ be a toric CY $3$-fold with transverse $A_n$ singularities.  Then to $\cX$ we can associate a \emph{web diagram}, a trivalent planar graphs $\Gamma=\{\text{Edges}, \text{Vertices}\}$ where vertices correspond to torus fixed points in $\cX$, edges correspond to torus invariant lines, and regions delineated by edges correspond to torus invariant divisors (see eg. \cite{bcy:otv} Appendix B).  We enrich the graph structure in several ways.  First, we orient each edge which is a choice of direction for each edge.  Let $n(e)$ denote the order of the isotropy on the line $C_e$ corresponding to an edge $e$.  We label the edges adjacent to each vertex $(e_1(v),e_2(v),e_3(v))$ so that
\begin{itemize}
\item if $v$ is adjacent to an edge $e$ with $n(e)>1$, then $e_3(v)=e$, and
\item the labels $(e_1(v),e_2(v),e_3(v))$ are ordered counterclockwise.
\end{itemize}
Moreover, we equip $\cX$ with a CY $\bC^*$ action which can be recorded by weights $(w_1(v), w_2(v), w_3(v))$ on the half-edges adjacent to each vertex $v$.

In order to prove the gluing formula, we must define factors $m_e$ and $k_e$ to each edge and its definition depends on whether $n(e)=1$ or $n(e)>1$.

\subsubsection{$n(e)=1$}

In this case, $C_e$ has possible orbifold structure only over the torus fixed points, so $C_e\cong \bP^1_{n_0,n_\infty}$.  Moreover, the normal bundle splits $\cN_{C_e/\cX}\cong\cN_r\oplus\cN_l$
where $\cN_r$ ($\cN_l$) corresponds to the normal bundle summand in the direction of the torus invariant divisor to the right (left) of $e$.  Let $p$ be a generic point on $C_e$ and $p_0,p_\infty$ the torus fixed points.  By the transverse $A_n$ condition, we can write
\[
\cN_r=\cO(m[p]-\delta_0[p_0]-\delta_\infty[p_\infty]),
\]
\[
\cN_l=\cO(m'[p]-\delta_0'[p_0]-\delta_\infty'[p_\infty])
\]
with the conditions
\[
n_0=1\Rightarrow \delta_0=\delta_0'=0,
\]
\[
n_0>1\Rightarrow \delta_0,\delta_0'\in\{0,1\} \text{ and } \delta_0+\delta_0'=1,
\]
and the analogous requirements at $p_\infty$.  Moreover, by the CY condition, we have
\[
m+m'-(\delta_0+\delta_0'+\delta_\infty+\delta_\infty')=-2.
\]
For each such oriented edge with $n(e)=1$, we define
\[
m_e:=m
\]
and
\[
k_e:=0.
\]

\subsubsection{$n(e)>1$}

In this case, $C_e$ is a $\bZ_n$ gerbe over $\bP^1$.  Define the gerbe $\cG_k$ by pullback
\[\begin{CD}
\cG_k   	@>>>          \cB\bC^*       \\
@VVV                             @VV\lambda\rightarrow\lambda^nV\\
\bP^1                  @>\cO(-k)>>      \cB\bC^*
\end{CD}\]
Then $C_e$ is isomorphic to some $\cG_k$ and the isomorphism is unique if we insist that the generator of $\bZ_{n(e)}$ acts on the fibers of $\cN_r$ by multiplication by $\xi_{n(e)}^{-1}$.  In this case, we have $\deg(\cN_r)=m\in\bZ+\frac{k}{n}$ and we define
\[
m_e:=m
\]
and
\[
k_e:=k.
\]
We also let $m'_e:=\deg(\cN_l)=-m_e-2$ and $\delta_0=\delta_0'=\delta_\infty=\delta_\infty'=0$.

\subsubsection{Gluing}
Let $\Lambda$ denote a GW \emph{edge assignment}, ie. an assignment of $n(e)$-partitions $\mu_e$ to each edge $e$.  Define
\[
\mu_i(v)=\begin{cases}
\mu_{e_i(v)} & \text{ if } n(e_i(v))=1 \text{ or if } n(e_i(v))>1 \text{ and }v \text{ is the initial vertex,}\\
\{\xi^{(\mu_e)_j^ik_e-i}\mu_j^i\} & \text{ if } n(e_i(v))>1 \text{ and } v \text{ is the final vertex.}
\end{cases}
\]
Recall that there is also an index $\alpha=(\alpha_+,\alpha_-)\in\{1,-1\}^2$ in definition \ref{framedvertex}. For each vertex $v$, we define
\[
\alpha_i(v)=\begin{cases}
1 & \text{ if } v \text{ is the initial vertex of $e_i(v)$,} \\
-1 & \text{ if } v \text{ is the initial vertex of $e_i(v)$.}
\end{cases}
\]
where $i=1,2$.
At each vertex $v$, we have an induced vertex term
\[
\tilde V_{\cX,\Lambda}^{\bullet,\alpha(v)}(v)=\begin{cases}
\tilde{V}_{\mu_1(v),\mu_2(v),\mu_3(v)}^{\bullet,(-\alpha_1(v),\alpha_2(v))}(w_1(v),w_2(v),w_3(v)) & \text{ if } v \text{ is the initial vertex of } e_3(v)\\
\tilde{V}_{\mu_2(v),\mu_1(v),\mu_3(v)}^{\bullet,(-\alpha_2(v),\alpha_1(v))}(w_2(v),w_1(v),w_3(v)) & \text{ if } v \text{ is the final vertex of } e_3(v).
\end{cases}
\]

\begin{remark}
The reason we need to be careful about orientations is because we have picked a particular isomorphism of the isotropy with $\langle \xi_n \rangle$ in the definition of $V$ and a particular orientation which depends on that choice.
\end{remark}

By the main gluing result in \cite{r:lgoa}, the GW potential of $\cX$ can be written
\begin{equation}\label{gwpot1}
GW(\cX)=\sum_{\Lambda}\prod_v \tilde V_{\cX,\Lambda}^{\bullet(v),\alpha(v)} \prod_e (-1)^{(m_e+\delta_0+\delta_\infty)|\mu_e|}z_{\mu_e}v_e^{|\mu_e|}
\end{equation}
where $v_e$ are formal parameters which track the degree of the map.

Applying Conjecture \ref{conjecture}, then applying orthogonality of characters along with the identity $\chi_\lambda(\{\xi^{(\mu_e)_j^ik-i}\mu_j^i\})=\xi^{-k\sum_i i|\lambda_i|}\overline{\chi_\lambda(\mu)}$, \eqref{gwpot1} becomes
\begin{equation}\label{gwpot2}
\sum_{\Lambda}\prod_v \tilde P_{\cX,\Lambda}^{\alpha(v)}(v) \prod_e \xi^{-k_e\sum_i i|(\lambda_e)_i|}(-1)^{(m_e+\delta_0+\delta_\infty)|\lambda_e|}v_e^{|\lambda_e|}
\end{equation}
where the sum is now over all $DT$ edge assignments and
\[
\tilde P_{\cX,\Lambda}^{\alpha(v)}(v)=\begin{cases}
\tilde{P}_{\lambda_1(v),\lambda_2(v),\lambda_3(v)}^{(-\alpha_1(v),\alpha_2(v))}(w_1(v),w_2(v),w_3(v)) & \text{ if } v \text{ is the initial vertex of } e_3(v)\\
\tilde{P}_{\lambda_1(v),\lambda_3(v),\lambda_2(v)}^{(-\alpha_2(v),\alpha_1(v))}(w_1(v),w_3(v),w_2(v)) & \text{ if } v \text{ is the final vertex of } e_3(v).
\end{cases}
\]


A tedious translation of the main result in \cite{bcy:otv}\footnote{We have taken into account a minor typo in the statement of the main theorem of \cite{bcy:otv} which we have confirmed with the authors.  In particular, in Theorem 12 of \cite{bcy:otv}, $H_{\nu'}$ should be replaced with $H_{\nu}$.} into our setup tells us that the reduced, multi-regular DT potential of $\cX$ is equal to \eqref{gwpot2} after negating $q$.  This finishes the proof.
\begin{remark}
We need two facts in the above translation of the main result in \cite{bcy:otv}. The first one is the symmetry property (\ref{dtsym}) of the DT vertex which eliminates the conjugate of the partition of the gerby leg in the gluing algorithm in \cite{bcy:otv}. The second fact is that the right (resp. left) normal bundle of an edge in \cite{bcy:otv} corresponds to the left (resp. right) normal bundle in our definition. However, we always have $(-1)^{n(e)(m_e+\delta_0+\delta_\infty)}=(-1)^{n(e)(m'_e+\delta'_0+\delta'_\infty)}$.
\end{remark}

\appendix

\section{Colored Young Diagrams and Loop Schur Functions}\label{sec:loop}

Through $n$-quotients (cf. \cite{rz:ggmv} Section 6), each $n$-partition $\lambda$ can be identified with a partition $\bar\lambda$ of $n|\lambda|$.  We color the Young diagram of $\bar\lambda$ with $n$-colors in the following way:
\begin{center}
\begin{ytableau}
*(yellow) & *(green) & *(white) & *(yellow)\\
*(white) & *(yellow) & *(green)\\
*(green) & *(white) & *(yellow)\\
*(yellow)
\end{ytableau}
\end{center}
with
\[
0\leftrightarrow\begin{ytableau}*(yellow)\end{ytableau}, \hspace{1cm} 1\leftrightarrow\begin{ytableau}*(green)\end{ytableau}, \hspace{.5cm}\text{and}\hspace{.5cm} 2\leftrightarrow\begin{ytableau}*(white)\end{ytableau}
\]

A \textit{semi-standard Young tableau} (SSYT) of $\bar\lambda$ is a numbering of the boxes so that numbers are weakly increasing left to right and strictly increasing top to bottom.   For each SSYT $T$ and $\square\in\bar\lambda$, we define the \textit{weight} $w(\square,T)$ to be the number appearing in that box.  To each $\bar\lambda$, $n$, and $T\in SSYT(\bar\lambda,n)$, we associate a monomial $q^T$ in $n$ infinite sets of variables $\{q_{i,j}|i\in\bZ_n,j\in\bN\}$:
\[
q^T:=\prod_{i=0}^{n-1}\prod_{\square\in\bar\lambda[i]}q_{i,w(\square,T)}.
\]
For example, to the SSYT (with $n=3$)
\[
T=\begin{ytableau}
*(yellow) 1& *(green) 1& *(white) 2& *(yellow) 4\\
*(white) 2& *(yellow) 3& *(green) 3\\
*(green) 4& *(white) 4& *(yellow) 6\\
*(yellow) 7
\end{ytableau}
\]
we associate the monomial
\[
q^T=q_{0,1}q_{0,3}q_{0,4}q_{0,6}q_{0,7}q_{1,1}q_{1,3}q_{1,4}q_{2,2}^2q_{2,4}.
\]
\begin{definition}
The \textit{loop Schur function} associated to $(\bar\lambda)$ is defined by
\[
s_{\lambda}(q_{i,j}):=\sum_{T\in SSYT(\bar\lambda,n)}q^T.
\]
\end{definition}

Denote by $s_{\lambda}(\bq)$ the function in $n$ variables obtained by making the substitution $q_{i,j}=q_i^j$ in $s_{\lambda}(q_{i,j})$.  The following result appears in both \cite{er:cpgisf} and \cite{n:hfgt}.

\begin{lemma}[\cite{er:cpgisf,n:hfgt}]\label{hookcontent}
\[
s_\lambda(\bq)=\frac{\prod_i \bq_i^{n_i(\bar\lambda)}}{\prod_{\square\in\bar\lambda}\left(1-\prod_i \bq_i^{h_i(\square)}\right)}.
\]
where $h_i(\square)$ denotes the number of color $i$ boxes in the hook defined by $\square$ and \[n_i(\bar\lambda):=\sum_k(k-1)(\text{\# of color i boxes in the $k$th row}).\]

\end{lemma}

\section{Rubber Integrals and Hurwitz Numbers}\label{sec:rubber}

Here we gather relevant information regarding rubber integrals which arise in our localization computations.  For $\nu$ and $\mu$ a pair of $n$-partitions, define the rubber integral generating function
\begin{equation}\label{eqn:rubber}
H_{\nu,\mu}(x,u):=\frac{1}{|\Aut(\nu)||\Aut(\mu)|}\sum_{g,\gamma}\int_{\Mbar_{g,\gamma}(\bP^1\times\cB\bZ_n;-\nu,\mu)//\bC^*}\psi_0^{2g-3+|\gamma|+l(\nu)+l(\mu)}u^{2g-2+l(\nu)+l(\mu)}\frac{x^\gamma}{\gamma!}
\end{equation}

For $n=1$, it is well known that the rubber integrals in \eqref{eqn:rubber} correspond to certain Hurwitz numbers.  In \cite{rz:ggmv}, this fact is generalized to compute the arbitrary rubber integrals in \eqref{eqn:rubber} in terms of wreath Hurwitz numbers.  In particular, we have
\[
\frac{r!}{|\Aut(\nu)||\Aut(\mu)|}\int_{\Mbar_{g,\gamma}(\bP^1\times\cB\bZ_n;\nu,\mu)//\bC^*}\psi_0^{2g-3+|\gamma|+l(\nu)+l(\mu)}=H_{\nu,\mu}^{g,\gamma}
\]
where $H_{\nu,\mu}^{g,\gamma}$ is the automorphism-weighted count of wreath Hurwitz covers $f:C\rightarrow\bP^1$ where the branch locus consists of a set of $|\gamma|+r+2$ fixed points (we fix the last two points at $0$ and $\infty$) and the maps satisfy the following conditions:
\begin{itemize}
\item The quotient $C/\bZ_n$ is a connected genus $g$ curve,
\item The monodromy around $0$ and $\infty$ is given by $\nu$ and $\mu$,
\item The monodromy around the branch point corresponding to $\gamma_i\in\gamma$ is given by the conjugacy class $\{\gamma_i,1,...,1\}$,
\item The monodromy around the $r$ additional branch points is given by the conjugacy class $\{2,1,...,1\}$.
\end{itemize}

We define

\begin{align*}
H^\bullet_{\nu,\mu}(x,u):&=\exp\left( \sum_{g,\gamma} H_{\nu,\mu}^{g,\gamma}\frac{u^r}{r!}\frac{x^\gamma}{\gamma!} \right)\\
&=\sum_{g,\gamma} H_{\nu,\mu}^{\chi,\gamma \bullet}\frac{u^r}{r!}\frac{x^\gamma}{\gamma!}
\end{align*}
where $H_{\nu,\mu}^{\chi,\gamma \bullet}$ is the wreath Hurwitz number with possibly disconnected covers.

By the Burnside formula, we compute
\[
H_{\nu,\mu}^{\chi,\gamma \bullet}=\sum_{|\lambda|=d}\left(f_T(\lambda)\right)^r \prod\left(f_{i}(\lambda)\right)^{m_i(\gamma)}\frac{\chi_{\lambda}(\mu)}{z_{\mu}}\frac{\chi_{\lambda}(\nu)}{z_{\nu}}
\]
where $f_T(\lambda)$ and $f_i(\lambda)$ are the \textit{central characters} defined by
\[
f_T(\lambda):=\frac{n|\lambda|(|\lambda|-1)\chi_{\lambda}(\{2,1,...,1\})}{2\cdot\text{dim}\lambda}
\]
and
\[
f_{i}(\lambda):=\frac{|\lambda|\chi_{\lambda}(\{\xi^i,1,...,1\})}{\text{dim}\lambda}.
\]

Therefore we obtain the following form for the generating function of wreath Hurwitz numbers:
\begin{equation}\label{hurgen}
H_{\nu,\mu}^\bullet(x,u)=\sum_{|\lambda|=d}\frac{\chi_{\lambda}(\mu)}{z_{\mu}}\frac{\chi_{\lambda}(\nu)}{z_{\nu}}e^{f_T(\lambda)u+\sum f_{i}(\lambda)x_i}.
\end{equation}

Using the fact that $\chi_\lambda(-\nu)=\overline{\chi_\lambda(\nu)}$, orthogonality of characters gives us the following relations:
\[
H_{\nu,\mu}^\bullet(x+y,u+v)=\sum_\sigma H_{\nu,\sigma}^\bullet(x,u)z_\sigma H_{-\sigma,\mu}^\bullet(y,v)
\]
and
\[
H_{\nu,-\mu}^\bullet(0,0)=\frac{1}{z_\mu}\delta_{\nu,\mu}.
\]

Purely for notational convenience regarding the particular way in which rubber integrals appear in our formulas, we also define the modified generating functions
\begin{align*}
\tilde H^\bullet_{\nu,\mu}(a):&=H^\bullet_{\nu,\mu}(a\xi_{2n}^{-1}x_1,...,a\xi_{2n}^{1-n}x_{n-1},\sqrt{-1}au)\\
&=\sum_{|\lambda|=d}\frac{\chi_{\lambda}(\mu)}{z_{\mu}}\frac{\chi_{\lambda}(\nu)}{z_{\nu}}e^{f_T(\lambda)\sqrt{-1}au+\sum f_{i}(\lambda)\xi_{2n}^{-i}ax_i}.
\end{align*}

Additionally, we recall from \cite{rz:ggmv} the following useful result.

\begin{lemma}\label{lem:central}
After the prescribed identification of variables,
\begin{equation}
e^{\frac{1}{n}\left(\sqrt{-1}f_T(\lambda)u+\sum \xi_{2n}^{-k}f_k(\lambda)x_k\right)}=\left(-\xi_{2n}\right)^{|\lambda|}\left(\prod_k \xi_n^{k|\lambda_k|}\right) \left(\prod_{(i,j)\in\bar\lambda} q_{j-i}^{j-i}\right)^{1/n}
\end{equation}
\end{lemma}

\bibliographystyle{alpha}
\bibliography{biblio}

\end{document}